\documentclass[12pt]{article}

\usepackage[latin1]{inputenc}
\usepackage[english]{babel}
\usepackage[T1]{fontenc}
\usepackage{lmodern}
\usepackage{graphicx, xcolor, mathrsfs}
\usepackage{amsfonts}
\usepackage{stmaryrd}
\usepackage[nice]{nicefrac}
\usepackage{amsmath,amsthm,amssymb}
\usepackage{enumerate}
\usepackage[colorlinks=true,citecolor=black,linkcolor=black,urlcolor=blue]{hyperref}
\usepackage[top=2.5cm, bottom=2.5cm, left=2.5cm , right=2.5cm]{geometry}
\usepackage{caption}
\usepackage{xifthen} 
\usepackage{bbm}
\usepackage{wrapfig}
\usepackage[numbers, square]{natbib}
\usepackage{comment}
\usepackage{float}
\usepackage{ifthen}

\usepackage{caption}
\newtheorem{thm}{Theorem}

\newtheorem{lem}{Lemma}
\newtheorem{clm}{Claim}

\newtheorem{cor}{Corollary}

\newcommand{\indicator}[1]{\mathbbm{1}_{\{ #1 \}}}

\def\Var{\mathop{\rm Var}\nolimits}
\def\Cov{\mathop{\rm Cov}\nolimits}

\def\Im{{\rm Im}}
\newcommand{\Ker}{\mathrm{Ker}}
\def\Spec{{\rm Spec}}

\newcommand{\R}{\mathbb R}
\newcommand{\Z}{\mathbb Z}
\newcommand{\N}{\mathbb N}
\newcommand{\E}[2][]{\ensuremath{\mathbb{E}_{#1} \left[#2 \right]}}

\newcommand{\Prob}[2][]{\ensuremath{\mathbb{P}_{#1} \left(#2 \right)}}

\newcommand{\eps}{\varepsilon}

\newcommand{\dg}[2]{\mathrm{deg}_{#2}(#1)}
\newcommand{\dgr}[1]{\mathrm{deg}(#1)}

\newcommand{\cE}{\mathcal {E}}

\newcommand{\cH}{\mathcal {H}}

\newcommand{\Bin}{\mathrm{Bin}}

\newcommand{\dgg}[2]{\mathrm{deg}_{#1} (#2)}
\newcommand{\lk}{\mathrm{lk}}

\newcommand{\C}[3]{
\ifthenelse{\isempty{#3}}
{\ifthenelse{\isempty{#2}}
{{\mathcal K}_{#1}} 
{{\mathcal K}_{#1} \left(#2 \right)}
} 
{\ifthenelse{\isempty{#2}}
{{\mathcal K}_{#1}^{(#3)}} 
{{\mathcal K}^{(#3)} \left(#2 \right)}
}
}

\newboolean{finite_case} 
\setboolean{finite_case}{false}   

\newboolean{graphs} 
\setboolean{graphs}{false}   

\newboolean{wrrt} 
\setboolean{wrrt}{true}   

\title{Algebraic and combinatorial expansion in random simplicial complexes}
\author{Nikolaos Fountoulakis\thanks{School of Mathematics, University of Birmingham, Edgbaston, Birmingham, UK. Supported by the EPSRC grant EP/P026729/1. \newline  E-mail: \texttt{n.fountoulakis@bham.ac.uk} and \texttt{michal.przykucki@gmail.com}.}\ \ and Micha{\l} Przykucki\footnotemark[1]}

\begin{document}
\maketitle
\abstract{In this paper we consider the expansion properties and the spectrum of the combinatorial Laplace operator of a $d$-dimensional Linial-Meshulam random simplicial complex, above the cohomological connectivity threshold.  We consider the spectral gap of the Laplace operator and the Cheeger constant as this was introduced by Parzanchevski, Rosenthal and Tessler (\emph{Combinatorica} 36, 2016).
We show that with high probability the spectral gap of the random simplicial complex as well as the Cheeger constant are both concentrated around the minimum co-degree of among all $d-1$-faces. Furthermore, 
we consider a generalisation of a random walk on such a complex and show that the associated conductance 
is with high probability bounded away from 0. 
}
\noindent  \bigskip
\\

{\small {\bf Keywords}: random simplicial complexes, Laplace operator, Cheeger constant, conductance }

{\small {\bf AMS Subject Classification 2020}: 05E45; 05C81; 55U10} 

\section{Introduction} 
In this paper will consider the expansion properties of a random binomial simplicial complex past the threshold 
for the cohomological connectivity. 
This model was introduced by Linial and Meshulam~\cite{ar:LinMesh2006} and it is 
a generalisation of the binomial random graph $G(n,p)$. 
Let $Y(n,p;d)$ denote the random $d$-dimensional simplicial complex on $[n]:=\{1,\ldots, n\}$ where all possible 
faces of dimension up to $d-1$ are present but each subset of $[n]$ of size $d+1$ becomes a face with probability $p=p(n) \in [0,1]$, independently of every other subset of size $d+1$.  When $d=1$, the model reduces to the binomial random graph on $[n]$ with edge probability equal to $p$.  

In their seminal paper, Linial and Meshulam~\cite{ar:LinMesh2006} considered the \emph{cohomological connectivity} of $Y (n,p;2)$, that is whether the cohomology group $H^1(Y(n,p;d);\Z_2)$  over $\Z_2$ on dimension $1$ is trivial. 
They discovered a threshold function for the 2-face probability $p$:
$$\lim_{n\to \infty} \Prob{H^1(Y(n,p;d);\Z_2) \ \mbox{is trivial}} = 
\begin{cases}
1, & \mbox{if $p = \frac{2\log n + \omega(n)}{n}$} \\
0, & \mbox{if $p = \frac{2\log n - \omega (n)}{n}$} 
\end{cases},
 $$ 
where $\omega: \N \to \R_+$ is an arbitrary function such that $\omega (n) \to \infty$ as $n\to \infty$ and 
$\log$ is the natural logarithm.

This generalises the classic theorem of Gilbert~\cite{ar:Gilbert59} and 
Erd\H{o}s and R\'enyi~\cite{ar:ErdosRenyi60} regarding the (graph) connectivity of $G(n,p)$ ($G(n,m)$, respectively). In particular, let $\omega : \N \to \R_+$
be a function such that $\omega (n) \to \infty $ as $n \to \infty$; 
if $p = \frac{\log n + \omega (n)}{n}$, then w.h.p. $G(n,p)$ is connected, whereas if $p = \frac{\log n - \omega (n)}{n}$, then w.h.p. $G(n,p)$ has an isolated vertex. 

The theorem of Linial and Meshulam was extended by Meshulam and Wallach~\cite{meshulamwallach} to any dimension $d \geq 2$, with $\Z_2$ replaced by any finite group $R$:
$$\lim_{n\to \infty} \Prob{H^{d-1}(Y(n,p;d);R) \ \mbox{is trivial}} = 
\begin{cases}
1, & \mbox{if $p = \frac{d\log n + \omega(n)}{n}$} \\
0, & \mbox{if $p = \frac{d\log n - \omega (n)}{n}$} 
\end{cases},
 $$  
 where $\omega$ is as above. 
 Linial and Meshulam~\cite{ar:LinMesh2006} asked whether this can be extended to $\Z$ (for random 
 $2$-complexes). 
 For $d=2$, {\L}uczak and Peled~\cite{ar:LuczakPeled2017} proved a hitting time  version of this result 
 considering the generalisation of the random graph process. This is a random process in which one constructs a 
 random simplicial complex on $n$ vertices with complete skeleton,  where in each step 
 a new 2-dimensional face is added, selected uniformly at random. 
 They showed that w.h.p. the $1$-homology group over $\Z$ becomes trivial the very moment all edges (1-faces) 
 lie in at least one 2-face. This was proved for $\Z_2$ by Kahle and Pittel~\cite{ar:KahlePittel2016}. 
 These results generalise the classic result of Bollob\'as and Thomason~\cite{ar:BolThom83} on the w.h.p. coincidence of the hitting times of connectivity with that of having minimum degree at least 1. 
For $d\geq3$,  Hoffman et al.~\cite{ar:HoffKahlePaq2017} provided a partial answer showing that the 1-statement holds for $H^{d-1} (Y(n,p;d);\Z)$ provided that $np \geq 80 d \log n$. 

Furthermore, Gundert and Wagner~\cite{gundert2016eigenvalues} showed that $H^{d-1} (Y(n,p;d); \R)$ is 
trivial provided that $np \geq C \log n$ where $C$ is a sufficiently large constant. Their approach was extended
by Hoffman, Kahle and Paquette~\cite{hoffman2012spectral} which extended this to $p$ such that $np\geq (1+\eps) d\log n$.  Very recently, Cooley, del Guidice, Kang and Spr\"ussel~\cite{ar:CdGKS2020} considered 
the cohomological connectivity of a generalised version of the Linial-Meshulam model in which random selection 
of faces takes place at all levels and not merely at the top level. 

In this paper, we will study the expansion properties of $Y(n,p;d)$ for $p$ as in the supercritical regime 
of the Linial-Meshulam-Wallach theorem. To be more precise, we will consider the case $np = (1+\eps)d\log n$, 
for an arbitrary fixed $\eps >0$,
and we will deduce sharp concentration results about the spectral gap of the (combinatorial) Laplace operator as well as the Cheeger constant of $Y(n,p;d)$.

For a sequence of events $(\cE_n )_{n \in \N}$, where $\cE_n$ is an event in the probability space of 
$Y(n,p;d)$, we say that they occur \emph{with high probability (w.h.p.)}, if $\Prob{\cE_n} \to 1$ as $n\to \infty$.  
(We will use the same term for events in the probability space of $G(n,p)$.)
If $X_n$ is a random variable defined on the  probability space of $Y(n,p;d)$ and $c \in \R$, we write 
$X_n = c (1+o_p(1))$, if $\Prob{|X_n - c| > \eps } \to 0$ as $n\to \infty$ - so loosely speaking 
$X_n \to c$ in probability as $n\to \infty$.

\subsection{Measures of expansion: the spectral gap and the Cheeger constant}

The definition and the use of the discrete Laplace operator in quantifying expansion properties of graphs 
 dates back to Alon and Milman~\cite{ar:AlonMilman85}. 
 For a graph $G=(V,E)$, the (combinatorial) Laplace operator $\Delta^+_G$ is defined as the difference 
 $D_G - A_G$, where $D_G=\mathrm{diag} ( \dgr{v})_{v\in V}$ and $A_G$ is the adjacency matrix of $G$. 
 In~\cite{ar:AlonMilman85}, Alon and Milman showed how the smallest positive 
 eigenvalue of the Laplace operator is linked to the structure of the graph as a metric space and, 
 in particular, to the distribution of distances between disjoint sets and the diameter of the graph.  
 The Laplace operator had been considered in graph theory earlier~\cite{tr:AndMor71,bk:Biggs74, ar:Fiedler73}
in relation to the number of spanning trees, the girth and connectivity of a graph. 

Let $\lambda (G)$ denote the smallest positive eigenvalue of $\Delta^+_G $ also known as the \emph{spectral gap}
of $\Delta^+_G$. It is a consequence of 
a more general result in~\cite{ar:AlonMilman85} that $\lambda (G)$ is bounded from above 
(up to some multiplicative constant) by the \emph{edge expansion} of $G$. 
This is defined as 
$$c (G):= \min_{A \subset V \ : \ 0 < |A|\leq |V|/2} \frac{e(A, V\setminus A)}{|A|},$$ where 
$e(A, V \setminus A)$ denotes the number of edges with one endpoint in $A$ and the other in $V\setminus A$;
it is called the \emph{Cheeger constant} of $G$. 
Lemma 2.1 in~\cite{ar:AlonMilman85}
implies that for any non-empty (proper) subset $A \subset V$ we have 
\begin{equation} \label{eq:graph_Cheeger} 
\lambda (G) \leq \frac{n \cdot e(A, V\setminus A)}{|A| |V\setminus A|} =h(A;G). 
\end{equation}
If $|A| \leq |V|/2$, then the above is at most $2c(G)$.
This is the discrete analogue of an inequality proved by Cheeger in~\cite{ar:Cheeger1970}. Setting 
$h(G) = \min_{A \ : \ 0 < |A| \leq |V|/2} h(A;G)$ one can complete the above inequality with a lower bound 
(proved by Dodziuk~\cite{ar:Dodziuk84}) and get
\begin{equation} \label{eq:Laplace-Cheeger-graphs}
\frac{h^2 (G)}{8 d_{\max} (G)} \leq \lambda (G) \leq h(G) \leq 2c(G), 
\end{equation}
where $d_{\max} (G)$ is the maximum degree of $G$. 
Another consequence of this result is that if $d_{\min} (G)$ denotes the minimum degree of $G$, 
then $\lambda (G) \leq \frac{|V|}{|V|-1} \delta (G)$ (this was also proved in~\cite{ar:Fiedler73}). 
Moreover, if $G$ is disconnected, then $\lambda (G)=0$. Thus, sometimes $\lambda (G)$ is called 
the \emph{algebraic connectivity} of $G$.  
Further properties of expander graphs in relation to the smallest positive eigenvalue of the Laplace 
operator were obtained by Alon in~\cite{ar:Alon86}. 

More generally, the spectrum of the Laplace operator of a graph also determines how the edges between subsets
of vertices are distributed. This is expressed through the well-known \emph{Expander-Mixing Lemma}. 
Roughly speaking, it states that if the entire non-trivial spectrum of the Laplace operator of a graph $G=(V,E)$ 
is close to $d$, then the density of edges between any two non-empty subsets $A, B \subset V$ is about $d/n$. Such an estimate about the number of edges within any given subset of $A \subset V$ 
was proved by Alon and Chung~\cite{ar:AlonChung88}, in the case where $G$ is a $d$-regular graph. 
It was generalised by Friedman and Pippenger in~\cite{ar:FriedmanPipp87}. 

The spectral gap of $\Delta^+(G(n,p))$ was considered recently by Kolokolnikov et al.~\cite{kolokolnikov2014algebraic}. 
Kolokolnikov et al.~\cite{kolokolnikov2014algebraic} considered $p = c \log n /n$, for $c > 1$ (that is, above the connectivity threshold), and showed that in this case w.h.p. 
\begin{equation}\label{eq:lambda_Gnp} 
|\lambda (G(n,p))- d_{\min} (G(n,p))| < C \sqrt{\log n}. 
\end{equation}

One of the results of our paper is to generalise this result to higher dimensions in the context of the Linial-Meshulam 
random simplicial complex $Y(n,p;d)$. 

\subsection{High dimensional Laplace operators}

Let $Y$ be a $d$-dimensional simplicial complex on a set $V$ with $|V|< \infty$. We let $Y^{(j)}$ denote  the 
set of $j$-dimensional faces in $Y$, that is, the faces containing exactly $j+1$ vertices, where $-1\leq j\leq d$. 
It is customary to set $Y^{(-1)} =\{ \varnothing \}$.  Also, note that $Y^{(0)} = V$. 
For abbreviation, we will be calling 
a $j$-dimensional face a \emph{$j$-face}. 
If all possible $j$-faces are present in $Y$, for $j \leq d-1$, then $Y$ is said to have \emph{complete skeleton}. 

Furthermore, for a $d-1$-face $\sigma$ in $Y$ its \emph{degree} $\dgr{\sigma}$ is the co-degree of $\sigma$ 
in $Y$, that is, $\dgr{\sigma} = |\{ v \in Y^{(0)} \ : \ \{v\} \cup \sigma \in Y^{(d)}  \}|$. 
We let $\delta (Y) = \min_{\sigma \in Y^{(d-1)}} \dgr{\sigma}$ 
be the minimum  co-degree of a $d-1$-face in $Y$. If $d=1$, then $\delta (Y)$ coincides with $d_{\min}$ of 
the associated graph.

In our context a $j$-face for $j\geq 2$ has two orientations which are 
the two equivalence classes of all permutations of its vertices which have the same sign. In other words, 
two permutations correspond to the same orientation if we can derive one from the other applying an even 
number of transpositions. 
If $\sigma$ is an oriented $j$-face, we denote by $\bar{\sigma}$ the opposite orientation of it, and by $Y^{(j)}_{\pm}$ we denote the set of oriented $j$-faces. Finally for an oriented $j$-face $\rho = [v_0,\dots, v_j]$, 
with $1 \leq j \leq d$, we let $\partial \rho$ denote the \emph{boundary of $\rho$}, which is the  set of oriented $j-1$-faces
$(-1)^i \rho \setminus v_i := (-1)^i [v_0,\ldots, v_{i-1},v_{i+1},\ldots, v_j]$, for $i=0,\ldots, j$.

The space of \emph{$j$-forms}, which we denote by $\Omega^{(j)}(Y;\R)$ is the vector space over $\R$ of all 
skew-symmetric functions on oriented $j$-faces. In other words, for $j \geq 2$ we define
$$\Omega^{(j)} (Y;\R) = \{ f: Y^{(j)}_{\pm} \to \R \ : \ f(\bar{\sigma}) = - f(\sigma), \ \forall \sigma \in Y^{(j)}_{\pm} \}, $$
whereas $\Omega^{(0)} (Y;\R)$ is just the set of all real-valued functions on $V$ and $\Omega^{(-1)} (Y;\R)$ is defined
as the set of all functions from $Y^{(-1)} = \{\varnothing\}$  to $\R$, which can be identified with $\R$.
The space $\Omega^{(j)} (Y;\R)$ is endowed with the inner product:
$$\langle f,g \rangle = \sum_{\sigma \in Y^{(j)}} w(\sigma) f(\sigma) g(\sigma), \ \mbox{for $f,g \in \Omega^{(j)} (Y;\R)$}, $$
where $w: Y \to (0,\infty)$ is a weight function.  

If $\sigma = [v_0,\ldots, v_j]$ is an oriented $j$-face and $v\in V$ not a member of $\sigma$, then  we set
$v\sigma = [v,v_0,\ldots, v_j]$. Furthermore, if $v\in V$ and $\sigma \in Y$, we write $v\sim \sigma$, if 
$\{ v\} \cup \sigma \in Y$ too. 
For $j=1,\ldots, d$, we define the $j$th \emph{boundary operator} $\partial_j : \Omega^{(j)} (Y;\R) \to \Omega^{(j-1)} (Y;\R)$: 
for $f \in \Omega^{(j)} (Y;\R)$ and $\sigma \in Y^{(j-1)}$ we set 
$$(\partial_j f) (\sigma) = \sum_{v:  v \sim \sigma} f(v\sigma). $$
It is well-known and easy to verify that for $j \geq 1$, we have $\partial_{j}\partial_{j+1} = 0$, whereby 
$\Im \partial_{j+1} \subseteq \Ker \partial_{j}$. We set $Z_j (Y)= \Ker \partial_{j}$ and $B_j (Y)= \Im \partial_{j+1}$. 
The set $Z_j (Y)$ is the set of $j$\emph{-cycles}, whereas $B_j(Y)$ is the set of $j$\emph{-boundaries}.

Note that both $Z_{j}, B_{j} \subseteq \Omega^{(j)} (Y;\R)$ and furthermore $(\Omega^{(j)} (Y;\R), \partial_j)_{j=1}^d$ is 
a chain complex. The group $H_j(Y;\R) = Z_j (Y)/B_j(Y)$ is the $j$th \emph{homology group} over $\R$. 

Similarly, one defines the $j$th \emph{coboundary operator} $\delta_{j}: \Omega^{(j)} \to \Omega^{(j+1)}$ as 
follows: if $\sigma= [v_0,\ldots, v_{j+1}]$ is an oriented $j+1$-face and $f \in \Omega^{(j)}$, then
$$(\delta_j f) (\sigma) = \frac{1}{w(\sigma)}\sum_{i=0}^{j+1} (-1)^i w(\sigma \setminus v_i) f(\sigma \setminus v_i), $$
where $\sigma \setminus v_i = [v_0,\ldots, v_{i-1}, v_{i+1},\ldots, v_{j+1}]$. 
It is not hard to show that $\Im \delta_{j-1} \subseteq \Ker \delta_{j}$
We set $Z^j (Y) = \Ker \delta_{j}$ (the set of closed $j$-forms) and $B^j (Y)= \Im \delta_{j-1}$ and 
$H^j (Y;\R) = Z^j(Y)/ B^j (Y)$, the $j$th \emph{cohomology group} over $\R$. 

A straightforward calculation shows that $\delta_{j-1}$ is the adjoint operator of $\partial_j$: 
for $f_1 \in \Omega^{(j-1)}(Y;\R)$ and $f_2 \in 
\Omega^{(j)}(Y;\R)$
\begin{equation} \label{eq:adjoint}
 \langle \delta_{j-1} f_1, f_2 \rangle = \langle f_1, \partial_j f_2\rangle.
 \end{equation}
 Note that $B^{j}(Y) = Z_{j}(Y)^{\perp}$ and $B_j(Y) = Z^j (Y)^\perp$ and
$$\Omega^{(d-1)} (Y;\R) = B^{d-1} (Y) \oplus Z_{d-1} (Y) = B_{d-1} (Y) \oplus Z^{d-1} (Y).$$ 

\subsubsection*{The Laplace operator and the spectral gap}
The  Laplace operator  associated with $Y$ is the operator 
$\Delta : \Omega^{(d-1)} (Y;\R) \to \Omega^{(d-1)} (Y;\R)$
defined as 
$$ \Delta = \Delta^+ + \Delta^-, $$
where 
$$\Delta^+ = \partial_d \delta_{d-1} \ (\mbox{upper Laplacian}), \ \mbox{and} \ \Delta^- = \delta_{d-2}\partial_{d-1} \ 
(\mbox{lower Laplacian}).$$
Note that~\eqref{eq:adjoint} implies that $\Ker \Delta^+ = Z^{d-1}(Y)$ whereas $\Ker \Delta^- = Z_{d-1}(Y)$.
The partial Laplace operators decompose the space $\Omega^{(d-1)} (Y; \R)$:
$$\Omega^{(d-1)} (Y;\R) = B^{d-1} (Y) \oplus Z_{d-1} (Y) = B_{d-1} (Y) \oplus Z^{d-1} (Y).$$ 

The subspace $\cH_{d-1} (Y)=\Ker \Delta$ is the called the space of \emph{harmonic $d-1$-forms}. 
Note that for any $f \in \cH_{d-1}(Y)$, the fact that $\delta_{j-1}$ is the adjoint of $\partial_j$ implies that 
$$\langle \partial_{d-1} f , \partial_{d-1} f \rangle, \langle \delta_{d-1} f , \delta_{d-1} f \rangle = 0,$$
whereby $f \in Z^{d-1} (Y) \cap Z_{d-1}(Y)$. Also, note that the definitions of $Z^{d-1}(Y), Z_{d-1}(Y)$ imply
that $Z^{d-1} (Y) \cap Z_{d-1}(Y) \subseteq \Ker \Delta$. Thus, in fact $\cH_{d-1} (Y)=Z^{d-1}(Y) \cap Z_{d-1} (Y)$.

The \emph{discrete Hodge decomposition} is due to Eckmann~\cite{ar:Eckmann44}:
\begin{eqnarray} \label{eq:Hodge}  
\Omega^{(d-1)} (Y;\R) &=&  B^{d-1}(Y) \oplus \underbrace{\cH_{d-1}(Y)  \oplus B_{d-1}(Y)}_{Z_{d-1}(Y)}.
\end{eqnarray}
It can be shown (see ~\cite{parzanchevski2016isoperimetric} p. 203) that 
\begin{equation} \label{eq:coh-isomorphism}
H^{d-1}(Y;\R) \cong \cH_{d-1}(Y) \cong H_{d-1} (Y;\R). 
\end{equation}
A quantity that is of interest is the \emph{spectral gap} $\lambda (Y)$ of a $d$-dimensional complex $Y$. 
This is defined as the minimal eigenvalue of the Laplacian or the upper Laplacian over $Z_{d-1}$ (the 
set of $d-1$-cycles). (Note that the two operators $\Delta^+$ and $\Delta$ coincide on $Z_{d-1}= \Ker \partial_{d-1}$.) We define
\begin{equation}\label{eq:gap-def}
\lambda (Y) := \min \Spec (\Delta |_{Z_{d-1}(Y)}) = \min \Spec (\Delta^+|_{Z_{d-1} (Y)}). 
\end{equation}
Horak and Jost~\cite{ar:JostHorrack} developed the theory of the Laplace operator for general weight 
functions. We will focus on special weighting schemes that give rise to generalisations of the well-studied
combinatorial Laplace operator as well as the normalised Laplace operator.

\subsubsection*{The combinatorial Laplace operator and the Cheeger constant} 

In the case where $w(\sigma ) = 1$ for all $\sigma \in Y$, the operator $\Delta^+$ is called the 
\emph{combinatorial (upper) Laplace operator} associated with $Y$. 
An algebraic manipulation can give explicitly the combinatorial Laplace operator: 
for $f \in \Omega^{(d-1)} (Y)$ and $\sigma = [v_0,\ldots, v_{d-1}] \in Y_{\pm}^{(d-1)}$, we have (as in (3.1) from
~\cite{parzanchevski2016isoperimetric}), taking $v_d = v$,
\begin{eqnarray*} 
(\Delta^+ f) (\sigma ) &=& \sum_{v  :  v\sigma \in Y_\pm^{(d)}} (\delta_{d-1}f) (v\sigma) = 
\sum_{v : v\sigma \in Y_\pm^{(d)}} \sum_{i=0}^d (-1)^i f(v \sigma \setminus v_i )  \\
&=& \sum_{v : v\sigma \in Y_\pm^{(d)}} \left( f(\sigma) - \sum_{i=0}^{d-1} (-1)^i f(v \sigma \setminus v_i ) \right) \\
&=& \dgr{\sigma} f(\sigma ) - \sum_{v : v\sigma \in Y_\pm^{(d)}} \sum_{i=0}^{d-1} (-1)^i f(v \sigma \setminus v_i ).
\end{eqnarray*}

If $Y$ is a finite simplicial complex with $|Y^{(0)}|=n$ vertices, we define 
\begin{equation} \label{eq:Cheeger_def}
h(Y) =\min_{|Y^{(0)}| = A_0 \uplus \cdots \uplus A_d} 
\frac{n \cdot |F(A_0,\ldots, A_d)|}{\prod_{i=0}^d |A_i|},
\end{equation}
where the minimum is taken over all partitions of $Y^{(0)}$ into $d+1$ non-empty parts $A_0, \ldots, A_d$
and $F(A_0, \ldots, A_d)$ is the set of $d$-faces with exactly one vertex in each one of the parts. 

The following theorem was proved by Parzanchevski, Rosenthal, and Tessler~\cite[Theorem 1.2]{parzanchevski2016isoperimetric}.
\begin{thm}
\label{thm:cheegerIneq}
For a finite complex $Y$ with a complete skeleton, 
$$\lambda(Y) \leq h(Y).$$
\end{thm}
Furthermore, the authors also derive an expander mixing lemma for complexes with complete skeleton. 
This assumption was removed by Parzanchevski~\cite{ar:Parzan17}. 

In~\cite{parzanchevski2016isoperimetric}, the authors discuss the existence of a lower bound in the spirit 
of the lower bound in~\eqref{eq:Laplace-Cheeger-graphs}. They observe (cf. Section 4.2 in~\cite{parzanchevski2016isoperimetric}), that a bound of the form $C \cdot h(Y)^m \leq \lambda (Y)$, 
for some $C, m>0$ cannot hold, providing as counterexample the minimal triangulation of a M\"obius strip, 
which has $\lambda (Y)=0$ but $h(Y) >0$. 

Parzanchevski et al.~\cite{parzanchevski2016isoperimetric} conjecture that an inequality of the form 
$C \cdot h(Y)^2 - c \leq \lambda (Y)$ should hold, where $C, c>$ depend on the maximum degree of any 
$d-1$-face of $Y$ as well as on the dimension of $Y$. 

Furthermore, they showed~\cite{parzanchevski2016isoperimetric} that for $D>0$ there exists $\gamma$ 
such that $\gamma = O(\sqrt{D})$, as  $D \to \infty$, such that 
if $np = D \cdot \log n$, then w.h.p. $\Spec (\Delta^+|_{Z_{d-1}}) \subset [(D-\gamma) \log n, (D+\gamma) \log n]$. This implies that w.h.p. $\lambda (Y(n,p;d)) = (D \pm O(\sqrt{D})) \log n$ if $np = D \log n$ and $D$ is sufficiently large.

Our results strengthen the latter, showing that if $np = (1+ \eps)d \log n$ and $\eps >0$ is fixed, the upper bound $\lambda (Y(n,p;d)) \leq h(Y(n,p;d))$ which follows from Theorem~\ref{thm:cheegerIneq} becomes tight in 
that $\lambda (Y(n,p;d)) = h(Y(n,p;d)) (1+o_p(1))$. 
Furthermore, we show that $\lambda (Y(n,p;d))  / np$ converges in probability as $n\to \infty$ 
to a certain constant which depends on $\eps$ and $d$.  
Recall that $\delta (Y(n,p;d))$ denotes the minimum co-degree among all $d-1$-dimensional faces of $Y(n,p;d)$. 
\begin{thm} \label{thm:Cheeger_mincodeg} Let $p = \frac{(1+\eps)d \log n}{n}$, where $\eps >0$ is fixed. There exists $C>0$ such that w.h.p. 
$$\delta (Y(n,p;d)) - C \sqrt{\log n}\leq  \lambda (Y(n,p;d) ) \leq h( Y(n,p;d)) \leq \delta (Y(n,p;d)).  $$
Furthermore, w.h.p.
$$ |\delta (Y(n,p;d)) - (1+\eps )a d \log n| < C \sqrt{\log n}, $$
where $a = a(\eps)$ is the solution to 
$$ \eps = (1+\eps) (1- \log a ) a. $$
\end{thm}

The above theorem not only strengthens the results of Parzanchevski et al.~\cite{parzanchevski2016isoperimetric} 
as far as the range of $p$ is concerned, but it also gives more precise asymptotics for large $\eps$. 
Remark 1.2 in~\cite{kolokolnikov2014algebraic} states that as $\eps \to \infty$, $a (\eps )= 1 - \sqrt{\frac{2}{(1+\eps)d}} + O\left( \frac{1}{\eps} \right)$. 
Hence, if we write $D = (1+\eps )d$, then for any $D$ sufficiently large we have w.h.p
$$\lambda (Y(n,p;d)), h(Y(n,p;d)) = \left( D - \sqrt{2 D} + O(1) \right)\log n + O(\sqrt{\log n}). $$

The proof of the above theorem has three parts. We start with the result on $\delta (Y(n,p;d))$ 
in Section~\ref{sec:min_codeg} (cf. Lemma~\ref{lem:minDegree}). Thereafter, in Section~\ref{sec:cheeger}
we show that $h(Y(n,p;d)) \leq \delta (Y(n,p;d))$  (cf. Theorem~\ref{thm:mainCheeger}). 
Hence, the upper bound on $\lambda (Y(n,p;d))$ follows from Theorem~\ref{thm:cheegerIneq}. 

For the lower bound on $\lambda (Y(n,p;d))$ in Section~\ref{sec:laplacians} we follow an approach similar to that of Gundert and Wagner~\cite{gundert2016eigenvalues}.
The lower bound is derived through a decomposition, essentially due to Garland~\cite{garland1973p}, of the Laplace operator $\Delta^+$ of a simplicial complex $Y$ into the sum of the (combinatorial) graph Laplace operators of the link graphs defined by the $d-2$-faces of $Y$. We show that the positive eigenvalues of these are bounded from below by $\delta (Y)$ and hence the lower bound in Theorem~\ref{thm:Cheeger_mincodeg}. 

However, this approximation incurs a term which involves the adjacency matrix of these link graphs. As we will see in Section~\ref{sec:laplacians}, in the case where $Y$ is $Y(n,p;d)$ these graphs are distributed as $G(n-d+1,p)$. 
At this point we use sharp results of Feige and Ofek~\cite{ar:FeigeOfek2005} to show that this term has no 
essential contribution. 

\subsection{The normalised Laplace operator} 
Under a weighting scheme where 
\begin{equation} \label{eq:normalised_weights}
w(\sigma ) = \begin{cases} 1 & \mbox{if $\sigma \in Y \setminus Y^{(d-1)}$} \\
\frac{1}{\dgr{\sigma}} & \mbox{if $\sigma \in Y^{(d-1)}$} 
 \end{cases}, 
 \end{equation}
 the operator $\Delta^+$ is called the \emph{normalised Laplace operator} associated with $Y$. 
 An explicit calculation as above  (cf. (2.6) in~\cite{ar:ParzRosen2017}) shows that for 
 $f \in \Omega^{(d-1)} (Y;\R)$
 and $\sigma = [v_0,\ldots, v_{d-1}] \in Y_\pm^{(d-1)}$ we have 
 \begin{eqnarray*} 
 (\Delta^+ f) (\sigma ) &=& f(\sigma) - \sum_{v : v\sigma \in Y_\pm^{(d)}} \sum_{i=0}^{d-1} \frac{(-1)^i f(v\sigma \setminus v_i)}{\dgr{v\sigma \setminus v_i}} \\
 &=& f(\sigma) - \sum_{\sigma' : \sigma'\sim \sigma} \frac{f(\sigma')}{\dgr{\sigma'}},
 \end{eqnarray*}
where for two oriented faces $\sigma, \sigma' \in Y_{\pm}^{(d-1)}$, we write $\sigma \sim \sigma'$ if there exists 
$\rho \in Y_\pm^{(d)}$ such that $\sigma, \overline{\sigma'} \in \partial \rho$. 

For graphs, the normalised Laplace operator acts on functions on the vertex set of a graph $G=(V,E)$ and is defined  
as $\mathcal{L}_G= D_G^{-1/2} \Delta^+ D_G^{-1/2} = I - D_G^{-1/2} A_G D_G^{-1/2}$, where $I$ is the identity operator. However, note that for $d=1$ the definition of $\Delta^+$ yields the operator $\Delta^+= I - D_G^{-1} A_G$. 
This has the same spectrum as $\mathcal{L}_G$, provided that $\dgr{\sigma}>0$, for all $\sigma \in Y^{(0)}$. 
Furthermore, the constant function on $V$ is an eigenfunction corresponding to eigenvalue 0, whereas all other eigenvalues are positive. 

Gundert and Wagner~\cite{gundert2016eigenvalues} showed that w.h.p. the non-trivial eigenvalues of 
the normalised Laplacian of $Y(n,p;d)$ are close to $1$, for $p$ such that $np \geq  C \log n$. This implies that 
$H^{d-1}(Y(n,p;d); \R)$ is trivial for such $p$.  
Hoffman, Kahle and Paquette~\cite{hoffman2012spectral} extended this argument for $p$ such that 
$n p \geq (1/2 + \delta) \log n$, showing that 
w.h.p. all non-trivial eigenvalues are within $C/\sqrt{np}$ from 1.

The argument of Gundert and Wagner~\cite{gundert2016eigenvalues} relies on proving the sharp concentration of the non-trivial eigenvalues of the normalised Laplacian of $G(n,p)$ around 1. Hence, the sharpening of Hoffman et al.~\cite{hoffman2012spectral} follows from their main result about the eigenvalues of the Laplacian of $G(n,p)$ for any $p$ such that $np \geq (1/2 + \delta) \log n$, for arbitrary fixed $\delta>0$.  
For the denser regime where $np =\Omega (\log^2 n)$, this was proved by Chung, Lu and Vu~\cite{ar:ChungLuVu2003}. However, for sparser regimes ($np$ bounded) this fact has been proved by Coja-Oghlan~\cite{ar:Coja-Oghlan2007} for the Laplace operator restricted on core of $G(n,p)$, although for $G(n,p)$
itself the spectral gap is $o_p(1)$.

\subsection{Random walks on $Y(n,p;d)$ and expansion}  

This part of the paper is motivated by the notion of a random walk on $Y$ introduced by Parzanchevski and Rosenthal~\cite{ar:ParzRosen2017}. 
This is in fact a random walk on $Y^{(d-1)}_\pm$ and, more precisely, on the
graph $(Y^{(d-1)}_\pm, E^{(d-1)}_\pm)$, where $\sigma \sigma' \in  E^{(d-1)}_\pm$ if and only if $\sigma \sim \sigma'$, for distinct $\sigma, \sigma'$. 

For example, if $Y$ is 2-dimensional complex, then this is a walk on the oriented edges (1-faces) of $Y$. 
If $[v,u]$ is such a face, then the walk can move to any edge $[v',u]$ or $[v,v']$ provided that $[v,u,v'] \in Y_\pm^{(2)}$.

One may consider the projection of such a walk on $Y^{(d-1)}$. 
For distinct $\sigma, \sigma' \in Y^{(d-1)}$, we also write $\sigma \sim \sigma'$, if there exists $\rho \in Y^{(d)}$ such that both $\sigma, \sigma' \subset \rho$. Suppose that for all $\sigma \in Y^{(d-1)}$ we have $\dgr{\sigma}>0$. 

If $(X_0,X_1,\ldots)$ denotes this Markov chain, 
then for any $n\geq 1$ the transition probabilities are 
$\Prob{X_n = \sigma' \mid X_{n-1}=\sigma} = \frac{1}{d \cdot \dgr{\sigma}}$, provided that 
$\sigma \sim \sigma'$; otherwise 
$\Prob{X_n = \sigma' \mid X_{n-1}=\sigma} = 0$. 

 In a more general setting, one may consider a 
\emph{$\gamma$-lazy} version of this random walk, for $\gamma \in (0,1)$, 
where $\Prob{X_n = \sigma \mid X_{n-1}=\sigma} = \gamma$
and $\Prob{X_n = \sigma' \mid X_{n-1}=\sigma} = \frac{1-\gamma}{d \cdot \dgr{\sigma}}$, for $\sigma \sim \sigma'$.

In this Markov chain, the stationary distribution on $Y^{(d-1)}$, denoted by $\pi$, is such that $\pi (\sigma)$ 
is proportional to $\dgr{\sigma}$. Note that $\sum_{\sigma \in Y^{(d-1)}}\dgr{\sigma} = (d+1) \cdot |Y^{(d)}|$.  
For any $\sigma \in Y^{(d-1)}$ we have $\pi (\sigma) = \frac{\dgr{\sigma}}{(d+1) \cdot |Y^{(d)}|}$.

We consider the mixing of such a random walk in the case where $Y$ is $Y(n,p;d)$
with $np \geq (1+\eps) d\log n$. In particular, we will consider the \emph{conductance} of this Markov chain
which we denote by $\Phi_{Y}$. 
First for any non-empty proper subset $S \subset Y^{(d-1)}$ we define
$$\Phi_Y(S) = \frac{Q(S,\overline{S})}{\pi (S) \pi (\overline{S})},$$
where $\overline{S}= Y^{(d-1)} \setminus S$ and $Q(S,\overline{S}) = \sum_{\sigma \in S} \sum_{\sigma' \in \overline{S} : \sigma' \sim \sigma} \pi(\sigma) \cdot \frac{1}{d \cdot \dgr{\sigma}}$ and 
$\pi (S) = \sum_{\sigma \in S} \pi (\sigma )$. 
The conductance $\Phi_Y$ is defined as 
$$\Phi_Y = \min_{S\subset Y^{(d-1)} :  0< \pi (S) \leq 1/2} \Phi_Y(S). $$
Note that if $np> (1+\eps) d \log n$, then a first moment argument shows that w.h.p. $\dgr{\sigma} > 0$, for all 
$\sigma \in Y^{(d-1)} (n,p;d)$.
We show that w.h.p. the conductance of $Y(n,p;d)$ is bounded away from 0. 
\begin{thm}
\label{thm:conductance}
Let $Y = Y(n,p;d)$ where $np = (1+\eps) d \log n$ and $\eps >0$ is fixed. 
Then there exists $\delta > 0$ such that w.h.p.
$$\Phi_{Y(n,p;d)} > \delta.$$
\end{thm}
The first author and Reed~\cite{ar:FounReed2008} showed the analogous result for largest connected component 
of $G(n,p)$ when $np = \Omega (\log n)$. 

We prove Theorem~\ref{thm:conductance} in Section~\ref{sec:conductance}. Its proof is based on a double counting argument that is facilitated by a weak version of the Kruskal-Katona theorem (cf. Theorem~\ref{thm:kk}).

\subsubsection{Tools: concentration inequalities}

%
In our proofs, we make use of the following variant of the Chernoff bounds (see \cite[Chapter~4]{Chernoffcite}).

\begin{lem}\label{feelthechern}
Let $p \in (0,1)$, $N \in \N,$ and $\eps > 0$. Then 
	\begin{equation}
	\label{eqn:ChernoffUpper}
		\Prob{\Bin(N,p) \ge (1+\eps)Np} \le e^{-\eps^2 Np / 3}
	\end{equation}
		and
	\begin{equation}
	\label{eqn:ChernoffLower}
		\Prob{\Bin(N,p) \le (1-\eps)Np} \le e^{-\eps^2 Np / 2}.
	\end{equation}
\end{lem}

\section{The minimum (co)degree of $Y(n,p;d)$} \label{sec:min_codeg}

The following lemma builds very strongly on the work of Kolokolnikov, Osting, and Von Brecht \cite{kolokolnikov2014algebraic}, who obtained very sharp bounds on the minimum vertex degree in $G(n,p)$, and its relation to the spectral gap of the graph just above the connectivity threshold, in particular, when $p = (1+\eps) \tfrac{\log n}{n}$ for $\eps > 0$. (More specifically, see Lemmas 3.3. and 3.4 in~\cite{kolokolnikov2014algebraic}.)

\begin{lem}
\label{lem:minDegree}
Let $p = (1+\eps) \frac{d \log n}{n}$, and let $a = a(\eps)$ denote the solution to
\begin{equation}
\label{eqn:aDefn}
\eps = (1+\eps)(1-\log a)a.
\end{equation}
Let $Y = Y(n,p;d)$ and let
\[
\delta (Y) = \min_{\sigma \in Y^{(d-1)}} |\{v \in [n] \setminus \sigma^{(0)} : \sigma \cup \{v\} \in Y^{(d)} \}|
\]
be the minimum co-degree of a $(d-1)$-dimensional face in $Y$. Then there exists a constant $C > 0$ such that w.h.p. we have that
\[
|\delta (Y) - (1+\eps) a d \log n | \leq C \sqrt{\log n}.
\]
\end{lem}
\begin{proof}
For a random variable $X$ following the binomial distribution $\Bin(n,p)$ and $c > 0$, let
\[
 f_n(p,c) = \Prob{X \leq c n p} = \sum_{i=0}^{\lfloor c n p \rfloor} \binom{n}{i} p^i (1-p)^{n-i}.
\]
For $c > 0$, set
\begin{equation}
\label{eqn:HDefn}
\cH(c) = c-c \log c - 1.
\end{equation}
Observe that by \eqref{eqn:aDefn} and \eqref{eqn:HDefn} we have that $(1+\eps)\cH(a(\eps)) = -1$.

It can be shown (see Lemma 3.3 in \cite{kolokolnikov2014algebraic}) that for $p = \Theta(\log n / n)$ 
there exist constants $c_1, c_2 > 0$ such that
\begin{equation}
\label{eqn:fBounds}
\frac{c_1 e^{np \cH(c)}}{\sqrt{np}} \leq f_n(p, c) \leq c_2 \sqrt{np} e^{np \cH(c)}.
\end{equation}

Recall that for a set $\sigma \subset [n]$ of $d$ vertices,  $\dgr{\sigma}$ denotes the number of $d$-dimensional faces in $Y$ containing $\sigma$. Clearly, $\dgr{\sigma}$ follows the binomial distribution $\Bin(n-d,p)$. 
The application of~\eqref{eqn:fBounds} to the random variable $\dgr{\sigma}$ yields: 
\begin{equation}
\label{eqn:factualBounds}
\frac{c_1 e^{np \cH(c)}}{2\sqrt{np}} \leq \frac{c_1 e^{(n-d)p \cH(c)}}{\sqrt{(n-d)p}} \leq f_{n-d}(p, c) \leq c_2 \sqrt{(n-d)p} e^{(n-d)p \cH(c)} \leq c_2 \sqrt{np} e^{np \cH(c)}.
\end{equation}

Furthermore, 
\[
a n p \pm \sqrt{np} = a(n-d)p \pm \sqrt{np} + a dp= \left ( a \pm \frac{\sqrt{np}}{(n-d)p} +\frac{a d}{n-d}\right ) (n-d)p.
\]
Since $a(\eps) \in (0,1)$ for all $\eps > 0$ (see Remark 1.3 in \cite{kolokolnikov2014algebraic}), by taking
\[
c_0^\pm = a \pm \frac{\sqrt{np}}{(n-d)p} + \frac{ad}{n-d}
\]
we can use \eqref{eqn:factualBounds} to see that
\begin{align*}
\Prob{\dgr{\sigma} \leq a n p - \sqrt{np} } & \leq c_2 \sqrt{np} e^{np \cH(c_0^-)} \\
  & = c_2 \sqrt{(1+\eps)d \log n} \cdot \exp ((1+\eps)d \log n \cH(c_0^-)).
\end{align*}
and
\begin{align*}
\Prob{\dgr{\sigma} \leq a n p + \sqrt{np} } & \geq \frac{c_1 e^{np \cH(c_0^+)}}{2\sqrt{np}} \\
  & = \frac{c_1 \exp((1+\eps)d \log n \cH(c_0^+))}{2\sqrt{(1+\eps)d \log n}}.
\end{align*}
Since both $\cH(c)$ and $\cH'(c) = -\log c$ are continuous and positive on $(0,1)$, we have that 
\[
\cH(c_0^\pm) = \cH(a) \pm \frac{(1+o(1))\cH'(a)}{\sqrt{(1+\eps)d \log n}}.
\]

We start by showing that w.h.p. we have $\delta (Y) \geq a n p - \sqrt{np}$. 
For a fixed subset $\sigma$ of size $d$ we have that
\begin{eqnarray*}
\lefteqn{\Prob{\dgr{\sigma} \leq a n p - \sqrt{np} }  \leq}\\
&& c_2 \sqrt{(1+\eps)d \log n} \exp \left ( (1+\eps)d \log n \left ( \cH(a) - \frac{(1+o(1))\cH'(a)}{\sqrt{(1+\eps)d \log n}} \right ) \right ) \\
 & &= c_2 \sqrt{(1+\eps)d \log n} \exp \left (- d \log n - \Theta \left ( \sqrt{\log n} \right ) \right ) \\
 & &= n^{-d} \exp \left (- \Theta \left ( \sqrt{\log n} \right ) \right ).
\end{eqnarray*}
Therefore the expected number of $d$-element sets $\sigma$ with $\dgr{\sigma} \leq a n p - \sqrt{np}$ is at most
\[
\binom{n}{d} n^{-d} \exp \left (- \Theta \left ( \sqrt{\log n} \right ) \right ) = o(1),
\]
and consequently w.h.p. we have $\delta (Y) \geq a n p - \sqrt{np}$.

To bound $\delta (Y)$ from above, we first find an upper bound on $\dgr{\sigma}$. 
Using that $(1+\eps) \cH(a)=-1$, we have
\begin{align*}
\Prob{\dgr{\sigma} \leq a n p + \sqrt{np} } & \geq \frac{c_1\exp \left ( (1+\eps)d \log n \left ( \cH(a) + \frac{(1+o(1))\cH'(a)}{\sqrt{(1+\eps)d \log n}}\right ) \right )}{2\sqrt{(1+\eps)d \log n}} \\
 & = \frac{1}{2\sqrt{(1+\eps)d \log n}} c_1\exp \left (- d \log n + \Theta \left ( \sqrt{\log n} \right ) \right ) \\
 & = n^{-d} \exp \left (\Theta \left ( \sqrt{\log n} \right ) \right ).
\end{align*}
Let $X_\sigma = \indicator{\dgr{\sigma} \leq a n p + \sqrt{np}}$ and let $N_0 = \sum_{\sigma \in Y^{(d-1)}} X_\sigma$ denote the number of $d$-element subsets $\sigma$ with $\dgr{\sigma} \leq a n p + \sqrt{np}$. Hence, letting $\mu = \E{|N_0|}$ we have
\begin{equation}
\label{eqn:mu}
\mu = \binom{n}{d} f_{n-d} (p, c_0^+) \geq \exp \left (\Theta \left ( \sqrt{\log n} \right ) \right ) \to \infty
\end{equation}
as $n \to \infty$.

By Chebyshev's inequality we then have
\begin{equation}
\label{eqn:chebyshev}
\Prob{|N_0-\mu| > \mu/2} \leq \frac{4\Var(N_0)}{\mu^2}.
\end{equation}
The co-degrees of two subsets $\sigma, \sigma'$ are independent whenever $|\sigma \cap \sigma'| \neq d-1$. Thus the variance of $N_0$ satisfies
\begin{align*}
\Var(N_0) & = \sum_{\sigma \in Y^{(d-1)}} \Var(X_\sigma) + \sum_{\sigma, \sigma' \in Y^{(d-1)} : |\sigma \cap \sigma'| = d-1} \Cov(X_\sigma, X_{\sigma'}) \\
& \leq \binom{n}{d} f_{n-d} (p, c_0^+) + n^{d+1} \Cov(X_\sigma, X_{\sigma'}),
\end{align*}
where $\sigma, \sigma'$ are two fixed sets satisfying $|\sigma \cap \sigma'| = d-1$. Since
\begin{equation}
\label{eqn:covariance}
\Cov(X_\sigma, X_{\sigma'}) = \Prob{X_\sigma = X_{\sigma'} = 1} - (f_{n-d} (p, c_0^+))^2,
\end{equation}
we focus on the value of $\Prob{X_\sigma = X_{\sigma'} = 1}$. Let $\dg{\sigma}{\setminus \sigma'}$ denote the number of $d$-dimensional faces that contain a $d$-subset $\sigma$ but do not contain the $d$-subset $\sigma'$. Using the law of total probability, conditioning on the presence or absence of the unique face that contains both $\sigma$ and $\sigma'$, and 
since $\dg{\sigma}{\setminus \sigma'}$ and $\dg{\sigma'}{\setminus \sigma}$ are identically distributed, 
we have
\begin{eqnarray} \label{eq:covs}
 \lefteqn{\Prob{X_\sigma = X_{\sigma'} = 1}  =} \nonumber \\
 & & \Prob {\dg{\sigma}{\setminus \sigma'} + 1 \leq a n p + \sqrt{np} } 
 \Prob {\dg{\sigma'}{\setminus \sigma} + 1 \leq a n p + \sqrt{np} } p  \nonumber \\
  & &+ \Prob{\dg{\sigma}{\setminus \sigma'} \leq a n p + \sqrt{np} } \Prob {\dg{\sigma'}{\setminus \sigma} \leq a n p + \sqrt{np} } (1-p) \nonumber  \\
  & &\leq \left[ \Prob {\dg{\sigma}{\setminus \sigma'} + 1 \leq a n p + \sqrt{np}}   \right ]^2 p + 
  \left[ \Prob{\dg{\sigma}{\setminus \sigma'} \leq a n p + \sqrt{np} } \right]^2.
\end{eqnarray}
In particular, the random variable $\dg{\sigma}{\setminus \sigma'}$ follows the binomial distribution 
$\Bin (n-d-1,p)$ and, thereby, it is stochastically dominated by $\Bin (n-d,p)$. 
So 
\begin{eqnarray*}
\lefteqn{\Prob {\dg{\sigma}{\setminus \sigma'} + 1 \leq a n p + \sqrt{np}} =} \\ 
& & \Prob {\Bin (n-d-1,p) \leq a n p + \sqrt{np}-1} \leq \Prob{\Bin (n-d,p) \leq a np + \sqrt{np}}
= f_{n-d} (p, c_0^+).
\end{eqnarray*}
For the second term in~\eqref{eq:covs}, we have 
\begin{align*}
 \Prob{\dg{\sigma}{\setminus \sigma'} \leq a n p + \sqrt{np} } & = \sum_{j=0}^{\lfloor a n p + \sqrt{np} \rfloor} \binom{n-d-1}{j} p^j (1-p)^{n-d-1-j} \\
  & \leq \sum_{j=0}^{\lfloor a n p + \sqrt{np} \rfloor} \binom{n-d}{j} p^j (1-p)^{n-d-1-j} \\
  & = \left ( 1 + \frac{p}{1-p} \right ) \sum_{j=0}^{\lfloor a n p + \sqrt{np} \rfloor} \binom{n-d}{j} p^j (1-p)^{n-d-j} \\
  &\stackrel{1-p> 1/2}{\leq}(1+2p) f_{n-d} (p, c_0^+).
\end{align*}

Hence for $n$ large enough we have
\[
 \Prob{X_\sigma = X_{\sigma'} = 1} \leq (p + (1+2p)^2) \left( f_{n-d} (p, c_0^+) \right )^2 
 \leq (1+6p) \left ( f_{n-d} (p, c_0^+) \right )^2,
\]
and consequently
\begin{equation}
\label{eqn:covIneq}
\Cov(X_\sigma, X_{\sigma'}) \leq 6p (f_{n-d} (p, c_0^+))^2.
\end{equation}
Thus by \eqref{eqn:chebyshev}, \eqref{eqn:covariance}, and \eqref{eqn:covIneq}, we obtain
\begin{align*}
\Prob{|N_0-\mu| > \mu/2} & \leq \frac{\mu + 6p n^{d+1} (f_{n-d} (p, c_0^+))^2}{\mu^2} \\
 & = \frac{1}{\mu} + O(n^{-d} \log n) = o(1),
\end{align*}
since we have $\mu \to \infty$. Thus with high probability we have that the minimum co-degree of a $d$-element set is at most $a n p + \sqrt{np}$ and the lemma holds.

\end{proof}

For $k \in \Z$, let $W_k(x)$ be the $k$-th branch of the Lambert $W$ function, defined as
\[
W_k(x) e^{W_k(x)} = x.
\]
We now discuss some further properties of the function $a(\eps)$ defined in \eqref{eqn:aDefn}.
This lemma shows that for small $\eps$ the function $a(\eps)$ is bounded by a linear function on $\eps$. 
We will use this bound in the next section.
\begin{lem} \label{lem:a_bound}
We have
\begin{equation}
 \label{eqn:aformula}
 a = a(\eps) = \exp \left ( 1 + W_{-1} \left ( - \frac{\eps}{e(1+\eps)} \right ) \right ) < \min \left \{1, 0.33 \eps \right \}
\end{equation}
for all $\eps > 0$.
\end{lem}
\begin{proof}
First, note that we can rewrite \eqref{eqn:aDefn} as
\[
\frac{\eps}{1+\eps} = (1-\log a)a = - e(\log a-1) \exp(\log a - 1),
\]
so we have that $\log a - 1 = W_k(-\frac{\eps}{e(1+\eps)})$ for some $k \in \Z$, and consequently that
\[
 a (\eps) = \exp \left ( 1 + W_{k} \left ( - \frac{\eps}{e(1+\eps)} \right ) \right ).
\]
Since only the $0$th and $(-1)$th branch of the Lambert $W$ function are real, we must have $k=0$ or $k=-1$. Finally, since for $-1/e<x<0$ we have $W_{-1}(x) < -1 < W_{0}(x)$ and we must have $a(\eps) \in (0,1)$, we see that we must have $k = -1$ for all $\eps > 0$.

We now move on to showing that $a(\eps) < \min \{1, 0.33\eps \}$. The bound $a(\eps) < 1$ follows from Remark 1.3 in \cite{kolokolnikov2014algebraic}. Hence we focus on the bound $a(\eps) < 0.33\eps$.

First, using the property that for any branch of the Lambert $W$ function and any $z \in (-e^{-1}, 0)$ we have $W'(z) = \frac{W(z)}{z(1+W(z))}$, and that $e^{W(z)} = z/W(z)$, we obtain
\begin{align*}
a'(\eps) & = \exp \left ( 1 + W_{k} \left ( - \frac{\eps}{e(1+\eps)} \right ) \right ) \frac{W_{k} \left ( - \frac{\eps}{e(1+\eps)} \right )}{- \frac{\eps}{e(1+\eps)}\left(1+W_{k} \left ( - \frac{\eps}{e(1+\eps)} \right )\right)} \frac{-1}{e(1+\eps)^2} \\
 & = - \frac{\exp \left ( W_{k} \left ( - \frac{\eps}{e(1+\eps)} \right ) \right )}{-\frac{\eps}{e(1+\eps)}} \frac{W_{k} \left ( - \frac{\eps}{e(1+\eps)} \right )}{\left(1+W_{k} \left ( - \frac{\eps}{e(1+\eps)} \right)\right)} \frac{1}{(1+\eps)^2} \\
 & = - \frac{1}{(1+\eps)^2 \left(1+W_{k} \left ( - \frac{\eps}{e(1+\eps)} \right) \right)}.
\end{align*}
It was shown by Chatzigeorgiou \cite{chatzigeorgiou2013lambert} that for $u > 0$ we have
\[
 W_{-1} \left ( -e^{-u-1} \right ) < -1 -\sqrt{2u} - \frac{2u}{3}.
\]
Taking $u = \log \frac{1+\eps}{\eps}$ leads to the bound
\[
a'(\eps) < \frac{1}{(1+\eps)^2 \left ( \sqrt{2 \log \frac{1+\eps}{\eps}} + \frac{2}{3} \log \frac{1+\eps}{\eps} \right )}.
\]
(Recall that $W_1(z) < -1$, so in absolute value the above inequality has the opposite sign.) Since $(1+\eps)^2$ is increasing in $\eps$, and  $\sqrt{2 \log \frac{1+\eps}{\eps}} + \frac{2}{3} \log \frac{1+\eps}{\eps}$ is decreasing in $\eps$, for $0 < \eps \leq 1/5$ we have
\[
a'(\eps) < \frac{1}{\left ( \sqrt{2 \log \frac{1.2}{0.2}} + \frac{2}{3} \log \frac{1.2}{0.2} \right )} < 0.33.
\]
Also, $a''(\eps)$ is positive for $\eps < \eps_0 < 0.189$, and negative for $\eps > \eps_0$, so the maximum value of $a'(\eps)$ is obtained for some $0< \eps < 0.189$, where $a'(\eps) < 0.33$. Since we have $W_{-1} \left ( - \frac{\eps}{e(1+\eps)} \right ) \to -\infty$ as $\eps \to 0$, and consequently $a(\eps) \to 0$ as $\eps \to 0$, we have that $a(\eps) < 0.33\eps$ for all $\eps > 0$.
\end{proof}

\section{Cheeger constant}
\label{sec:cheeger}

In this section we consider the measure of expansion of a simplicial complex which is called its \emph{Cheeger constant} and was defined in~\eqref{eq:Cheeger_def}.
%
As the main result in this section we prove the following theorem.

\begin{thm}
\label{thm:mainCheeger}
Let $p = (1+\eps)  \frac{d \log n}{n}$ and let $Y = Y(n,p;d)$ and let $a=a(\eps)$ be as in~\eqref{eqn:aDefn}. There exists a positive constant $C > 0$ such that w.h.p. we have
\begin{equation}
\label{eqn:mainCheeger}
(1+\eps) a d \log n  - C\sqrt{\log n} \leq h(Y) \leq \delta (Y) \leq (1+\eps) a d \log n  + C\sqrt{\log n} .
\end{equation}
\end{thm}
\begin{proof}
The upper bound on $h(Y)$ follows immediately from Lemma \ref{lem:minDegree}. Indeed, if $A = \{a_0, \ldots, a_{d-1} \}$ is a $d$-element set with the minimum co-degree, then w.h.p. taking $A_i = \{ a_i \}$ for $0 \leq i \leq d-1$ and $A_d = [n] \setminus A$ gives us a partition with the desired value of 
$|F(A_0, A_1, \ldots, A_d)| = \delta (Y)$.
Thus, $h(Y) \leq \delta (Y)$.

So now we focus on lower-bounding $h(Y)$. First, observe that for a given set of values of $|A_0|, |A_1|, \ldots, |A_d|$, the value of $|F(A_0, A_1, \ldots, A_d)|$ follows the binomial distribution $\Bin (\prod_{i=0}^d |A_i|, p)$.

Without loss of generality let us assume that $|A_0| \leq |A_1| \leq \ldots \leq |A_d|$; note that this implies that $|A_d| \geq n/(d+1)$. We shall consider three possible cases, depending on the size of the second largest set $A_{d-1}$. 

First, let us assume that $|A_{d-1}| \geq n/(\log n)^{1/2}$, and set $|A_d| = \alpha n$ for some $\alpha \in [1/(d+1),1]$. In this case by \eqref{eqn:ChernoffLower} we have
\begin{align*}
& \Prob{n \cdot |F(A_0, A_1, \ldots, A_d)| \leq (1+\eps) \prod_{i=0}^d |A_i| \log n} \\
 & \qquad = \Prob{\Bin \left (\alpha n \prod_{i=0}^{d-1} |A_i|, (1+\eps)  \frac{d \log n}{n} \right ) \leq (1+\eps) \alpha \prod_{i=0}^{d-1} |A_i| \log n} \\
 & \qquad \leq \exp \left ( \left ( \frac{(1+\eps)d - (1+\eps)}{(1+\eps)d} \right ) ^2 (1+\eps) \alpha d  \log n \prod_{i=0}^{d-1} |A_i| / 2 \right ) \\
 & \qquad = \exp \left (- \Omega \left (n (\log n)^{1/2} \right ) \right ),
\end{align*}
where the last equality follows from $|A_{d-1}| \geq n/(\log n)^{1/2}$. Since there are at most $(d+1)^n$ partitions of $V$ into $d+1$ disjoint sets, by the union bound we see that with probability $1-o(1)$ every partition with $|A_{d-1}| \geq n/(\log n)^{1/2}$ has
\[
\frac{n \cdot |F(A_0, A_1, \ldots, A_d)|}{\prod_{i=0}^d |A_i|} > (1+\eps) \log n > (1+\eps) a(\eps) \log n,
\]
where the last inequality follows from $a(\eps) < 1$.

Now, assume that $C(d, \eps) \leq |A_{d-1}| < n/(\log n)^{1/2}$ for some constant $C(d, \eps) > 0$ to be determined later. Note that by the assumption that $|A_0| \leq |A_1| \leq \ldots \leq |A_d|$, this implies that $|A_d| = (1-o(1))n$. Observe that, bounding crudely, there are at most $n^{d}$ possible choices of the sizes $|A_0|, |A_1|, \ldots, |A_d|$ (after selecting $|A_0|, |A_1|, \ldots, |A_{d-1}|$, we have $|A_d| = n - (|A_0| + |A_1| + \ldots + |A_{d-1}|)$). For each such choice of $|A_0|, |A_1|, \ldots, |A_d|$, there are at most $n^{|A_0|+|A_1|+ \ldots + |A_{d-1}|}$ possible choices of partitions. Hence, in total, we will want to show that the probability that for a given partition our lower bound on the Cheeger constant does not hold is
\[
o \left ( n^{-(d+|A_0|+|A_1|+ \ldots + |A_{d-1}|)} \right ).
\]
We will show that in this case with probability $1-o \left ( n^{-(d+|A_0|+|A_1|+ \ldots + |A_{d-1}|)} \right )$ we have
\[
\frac{n \cdot |F(A_0, A_1, \ldots, A_d)|}{\prod_{i=0}^d |A_i|} \geq (1+\eps) \min \left \{1, 0.33\eps \right \} \log n 
\stackrel{Lemma~\ref{lem:a_bound}}{\geq} (1+\eps) a(\eps) \log n.
\]

Since $|A_d| = (1-o(1))n$, we have
\begin{align*}
& \Prob{n \cdot |F(A_0, A_1, \ldots, A_d)| \leq (1+\eps) \min \left \{1, 0.33\eps \right \} \prod_{i=0}^d |A_i| \log n} \\
 & \qquad \leq \Prob{\Bin \left ((1-o(1)) n \prod_{i=0}^{d-1} |A_i|, (1+\eps)  \frac{d \log n}{n} \right ) \leq (1+\eps) \min \left \{1, 0.33\eps \right \} \prod_{i=0}^{d-1} |A_i| \log n} \\
 & \qquad \leq \exp \left ( - \left ( \frac{(1-o(1))(1+\eps)d - (1+\eps) \min \left \{1, 0.33\eps \right \}}{(1-o(1))(1+\eps)d} \right ) ^2 (1+\eps) d  \log n \prod_{i=0}^{d-1} |A_i| / 2 \right ) \\
 & \qquad = \exp \left ( - \left ( \frac{d - (1+o(1)) \min \left \{1, 0.33\eps \right \}}{d} \right ) ^2 (1+\eps) d  \log n \prod_{i=0}^{d-1} |A_i| / 2 \right ).
\end{align*}
We note that, since $d \geq 2$, if $\eps > 3.01$ then for $n$ large enough we can bound
\begin{align*}
 \left ( \frac{d - (1+o(1)) \min \left \{1, 0.33\eps \right \}}{d} \right ) ^2 (1+\eps) & \geq \left ( \frac{d - (1+o(1))}{d} \right ) ^2 (1+\eps) \\
  & \geq \left ( \frac{1-o(1)}{2} \right ) ^2 \cdot 4.01 > 1.001.
\end{align*}
On the other hand, if $\eps \in (0,3.01]$ then we have that
\begin{align*}
 \left ( \frac{d - \min \left \{1, 0.33\eps \right \}}{d} \right ) ^2 (1+\eps) - 1 & = \left ( 1 - 0.33\eps/d \right ) ^2 (1+\eps) - 1 \\
  & \geq \left ( 1 - 0.33\eps/2 \right ) ^2 (1+\eps) - 1.
\end{align*}
As a polynomial in $\eps$, this is positive on $(0,3.04)$. So for any $\eps \in (0,3.01]$, if $n$ is large enough then we also have $\left ( \frac{d - (1+o(1)) \min \left \{1, 0.33\eps \right \}}{d} \right ) ^2 (1+\eps) = 1+\delta_\eps > 1$.

Since for all $0 \leq i \leq d-2$ we have $|A_i| \geq 1$, we can set $|A_0|+|A_1|+ \ldots + |A_{d-1}| = d-1+M$ for some $M \geq C(d, \eps)$. Given that, the minimum value of $\prod_{i=0}^{d-1} |A_i|$ is attained when $|A_0| = \ldots = |A_{d-2}| = 1$ and $|A_{d-1}| = M$. This gives
\begin{align*}
\min \{ 1.001, 1+\delta_\eps \} \frac{d}{2} \prod_{i=0}^{d-1} |A_i| & \geq \min \{ 1.001, 1+\delta_\eps \} \frac{dM}{2} \\
 & \geq \min \{ 1.001, 1+\delta_\eps \} M \geq 2d + M
\end{align*}
when $M \geq \max \{2000d, 2d/\delta_\eps\}$. Hence we obtain
\begin{eqnarray*}
 &&\Prob{n \cdot |F(A_0, A_1, \ldots, A_d)| \leq (1+\eps) \min \left \{1, 0.33\eps \right \} \prod_{i=0}^d |A_i| \log n}\leq \\
 & & \hspace{5cm} \exp \left ( - \left (d + 1 +\sum_{i=0}^{d-1} |A_i| \right ) \log n\right ).
\end{eqnarray*}
Hence, the union bound completes the proof for this case.
%

The final remaining case is when $|A_{d-1}| < C(d, \eps)$ for $C(d, \eps) > 0$ constant. In this case we can use Lemma \ref{lem:minDegree}. Since we know that w.h.p. the minimum co-degree of a $d$-element set is at least $(1+\eps)a(\eps)d \log n - C \sqrt{\log n}$, 
it follows that  w.h.p. for any selection of elements $a_0 \in A_0, \ldots, a_{d-1} \in A_{d-1}$ there are at least $(1+\eps)a(\eps)d \log n - C \sqrt{\log n} - d C(d, \eps)$ faces $\{a_0, \ldots, a_{d-1}, a_d \}$ with $a_d \in A_d = [n] \setminus \bigcup_{i=0}^{d-1} A_i$. Thus w.h.p. for any such partition we have
\begin{align*}
\frac{n \cdot |F(A_0, A_1, \ldots, A_d)|}{\prod_{i=0}^d |A_i|} & \geq \frac{n \prod_{i=0}^{d-1} |A_i| ((1+\eps)a(\eps)d \log n - C \sqrt{\log n} - d C(d, \eps))}{n \prod_{i=0}^{d-1} |A_i|} \\
& \geq (1+\eps)a(\eps)d \log n - 2C \sqrt{\log n}.
\end{align*}
This completes the proof of Theorem \ref{thm:mainCheeger}.
\end{proof}

\section{Algebraic expansion}
\label{sec:laplacians}

Recall from \eqref{eq:gap-def} that for a $d$-dimensional simplicial complex $Y$ with complete skeleton, the \emph{spectral gap} $\lambda(Y)$ is the minimal eigenvalue of the upper Laplacian on $(d-1)$-cycles, i.e.,
\[
\lambda(Y) = \min \mathrm{Spec} \left ( \Delta^+ |_{Z_{d-1}} \right ).
\]

As our main result in this section, we prove the following theorem about the spectral gap of $Y(n,p;d)$.

\begin{thm}
\label{thm:laplacians}
Let $Y = Y(n,p;d)$ for $d \geq 2$ and $p = \frac{(1+\eps)d \log n}{n}$, where $\eps >0$. 
There exists a constant $C > 0$ such that w.h.p. the spectral gap of $Y$ satisfies
\begin{equation}
\label{eqn:spectral}
 \delta (Y)  - C \sqrt{\log n} \leq \lambda(Y) \leq \delta (Y).
\end{equation}
\end{thm}
\begin{proof}
The upper bound in \eqref{eqn:spectral} follows immediately from Theorem~\ref{thm:mainCheeger} and through Theorem \ref{thm:cheegerIneq}. Recall that the latter states that 
$\lambda (Y) \leq h(Y)$. Furthermore, at the beginning of the proof of Theorem~\ref{thm:mainCheeger}, we observed that $h(Y) \leq \delta (Y)$.  

We now focus on the lower bound on $\lambda(Y)$. We will give a lower bound on $\langle \Delta^+ f, f\rangle$ for $f \in Z_{d-1}(Y)$ which relies on a decomposition of $\Delta^+$ into the Laplace operators 
of the link graphs of all $d-2$-faces of $Y$. 

For a $(d-2)$-dimensional face $\tau$ in $Y^{(d-2)}$, let $\lk \tau$ be the \emph{link graph} of face $\tau$, i.e., the graph with vertex set $Y^{(0)} \setminus \tau$, with $u,v \in Y^{(0)} \setminus \tau$ forming an edge if $\tau \cup \{u,v\} \in Y^{(d)}$. Note that when $\tau$ is a $(d-2)$-dimensional face in $Y(n,p;d)$, then $\lk \tau$ is a random graph with the $G(n-d+1,p)$ distribution.

For $\sigma \in Y^{(d-1)}$, we let $\dgg{\tau}{\sigma} = |\{\sigma' \in Y^{(d-1)} \ : \ \sigma' \sim \sigma, \ \tau \subset \sigma' \}|$. For convenience we write $\Omega^{(j)}$ for $\Omega^{(j)} (Y;\R)$.
We define the localised (on $\tau \in Y^{(d-2)}$) upper Laplace operator 
$\Delta_\tau^+ : \Omega^{(d-1)} \to \Omega^{(d-1)}$ ;
for any $f \in \Omega^{(d-1)}$ and any $\sigma \in Y^{(d-1)}$ we set
\begin{equation} \label{eq:def_loc_Lapl}
(\Delta_\tau^+ f)(\sigma) = \begin{cases} 
\dgg{\tau}{\sigma} - \sum_{\sigma' : \sigma' \sim \sigma, \tau\subset \sigma'} f(\sigma') &
\mbox{if $\tau \subset \sigma$} \\ 
0 & \mbox{$\tau \not \subset \sigma$}
\end{cases}.
\end{equation}
Furthermore, for a form $f \in \Omega^{(d-1)}$ and a face $\tau \in Y^{(d-2)}$, we define 
$f_\tau : (\lk \tau)^{(0)} \to \R$ as $f_\tau (v) = f(v\tau)$. 
The operator $\Delta_\tau^+$ has the same effect as the Laplace operator associated with 
$\lk \tau$ (see 2. in the following theorem). 
This is made precise in the following result by Garland \cite{garland1973p} (we state it as in Lemma 4.2 
in~\cite{parzanchevski2016isoperimetric}). 
\begin{thm}
\label{thm:garland}
Let $X$ be a $d$-dimensional simplicial complex and $f \in \Omega^{(d-1)}$.
Then the following hold. 
\begin{enumerate}
\item[1.] $\Delta^+  = \sum_{\tau \in X^{(d-2)}} \Delta_\tau^+ - (d-1)D$, where 
$(Df)(\sigma) = \dgr{\sigma} f(\sigma)$.
\item[2.] For any $\tau \in X^{(d-2)}$, \ $\langle \Delta_{\tau}^+ f,f\rangle = \langle\Delta^+_{\lk \tau} f_\tau , f_\tau \rangle$.
\item[3.] If $f \in Z_{d-1}(X)$, then $f_\tau \in Z_0 (\lk \tau)$.
\item[4.] $\sum_{\tau \in X^{(d-2)}} \langle f_\tau, f_\tau\rangle = d \langle f,f \rangle$.
\end{enumerate}
\end{thm}
For $f\in Z_{d-1} (Y)$ with $f \not = 0$, we seek a lower bound on $\langle \Delta^+ f,f\rangle$. 
We will use the decomposition of $\Delta^+$ 
in terms of the local Laplace operators as in 1. of the above theorem. To make use of this, we will rewrite 
the operator $D$ as a sum of localised versions of it. In particular, we write 
$D_{\lk \tau} : \Omega^{(0)} (\lk \tau;\R) \to 
\Omega^{(0)} (\lk \tau;\R)$ for the operator such that for $f' \in \Omega^{(0)} (\lk \tau;\R)$  and $v \in \lk \tau^{(0)}$
we have $(D_{\lk \tau} f')(v) = \dgg{\tau}{v\tau} f'(v)$. So, in particular, if $f \in \Omega^{(d-1)}$, then 
$(D_{\lk \tau} f_\tau )(v) = \dgg{\tau}{v\tau} f(v\tau)$.
The following holds. 
\begin{clm} \label{claim:D_decomp}
For any $f \in \Omega^{(d-1)} (Y)$, we have 
$$ \langle Df,f\rangle = \frac{1}{d} \sum_{\tau \in Y^{(d-2)}} \langle D_{\lk \tau}f_\tau,f_\tau \rangle. $$
\end{clm}
\begin{proof} Let $f \in \Omega^{(d-1)} (Y)$. We write
\begin{eqnarray}
\langle Df,f \rangle &=& \sum_{\sigma \in Y^{(d-1)}} \dgr{\sigma} f^2(\sigma) = 
\frac{1}{d} \sum_{\tau \in Y^{(d-2)}} \sum_{v \in \lk \tau^{(0)}} \dgr{v\tau} f^2 (v\tau) \nonumber \\ 
&=& \frac{1}{d} \sum_{\tau \in Y^{(d-2)}} \sum_{v \in \lk \tau^{(0)}} \dgg{\tau}{v} f_\tau ^2 (v) \nonumber \\
&=& \frac{1}{d} \sum_{\tau \in Y^{(d-2)}} \langle D_{\lk \tau} f_\tau, f_\tau \rangle. \nonumber
\end{eqnarray}
\end{proof}

Note that if $A_{\lk \tau}$ is the adjacency matrix of $\lk \tau$, then 
$$ \Delta^+_{\lk \tau} = D_{\lk \tau} - A_{\lk \tau}. $$
Hence using Theorem~\ref{thm:garland} parts 1. and 2., for any $f \in \Omega^{(d-1)}$ we can write
\begin{eqnarray}  
\langle \Delta^+ f, f\rangle &\stackrel{Thm~\ref{thm:garland}~1.}{=}&  \sum_{\tau \in Y^{(d-2)}}\langle \Delta_\tau^+f,f\rangle - (d-1)\langle Df,f\rangle =  \nonumber \\
&\stackrel{Thm~\ref{thm:garland}~2., \ Claim~\ref{claim:D_decomp}}{=}& \sum_{\tau \in Y^{(d-2)}} 
\langle \Delta_{\lk \tau}^+ f_\tau,f_\tau \rangle - \frac{d-1}{d} \sum_{\tau \in Y^{(d-2)}} \langle D_{\lk \tau}f_\tau,f_\tau\rangle \nonumber \\ 
&=& \sum_{\tau \in Y^{(d-2)}} \langle (D_{\lk \tau} - A_{\lk \tau}) f_\tau,f_\tau\rangle - \frac{d-1}{d} \sum_{\tau \in Y^{(d-2)}} \langle D_{\lk \tau}f_\tau,f_\tau\rangle  \nonumber \\ 
&=& \sum_{\tau \in Y^{(d-2)}} \left(\langle (D_{\lk \tau} - A_{\lk \tau}) f_\tau,f_\tau\rangle  -   
\frac{d-1}{d} \langle D_{\lk \tau}f_\tau,f_\tau\rangle \right) \nonumber \\
&=& \sum_{\tau \in Y^{(d-2)}} \left(\frac{1}{d} \langle D_{\lk \tau} f_\tau, f_\tau\rangle - \langle A_{\lk \tau} f_\tau, f_\tau \rangle \right).
\label{eq:Laplace_decomp}
\end{eqnarray}
But note that 
$$ \langle D_{\lk \tau} f_\tau, f_\tau\rangle = \sum_{v \in \lk \tau^0} \dgg{\tau}{v\tau} f_\tau (v)^2 \geq \delta (Y) \langle f_\tau, f_\tau \rangle. $$
Thereby, 
\begin{eqnarray}
\frac{1}{d}  \sum_{\tau \in Y^{(d-2)}} \langle D_{\lk \tau} f_\tau, f_\tau \rangle &\geq&  \delta(Y) \cdot \frac{1}{d}
\sum_{\tau \in Y^{(d-2)}} \langle f_\tau, f_\tau\rangle  \stackrel{Thm~\ref{thm:garland}~4.}{=} \delta (Y) \frac{d}{d} \langle f,f\rangle \nonumber \\
&=& \delta(Y) \langle f,f \rangle.  \label{eq:Ds_lowerbound}
\end{eqnarray}
Now, we will give an upper bound on $\sum_{\tau \in Y^{(d-2)}} \langle A_{\lk \tau} f_\tau, f_\tau \rangle$, for $Y=Y(n,p;d)$ and 
$f\in Z_{d-1}$. 
As we observed above, $\lk \tau$ is distributed as a $G(n- d+1,p)$ random graph for any 
$\tau \in Y^{(d-2)} (n,p;d)$. Of course, these random graphs are not independent. 

Hence, the simplest way to bound from above this sum is to give an upper bound on $\langle A_{\lk \tau} f_\tau, f_\tau\rangle$
that holds with probability $1- o(n^{-(d-1)})$ and apply the union bound over all $\tau \in Y^{(d-2)} (n,p;d)$, which 
there are $O(n^{d-1})$ of them. 

To this end, we will use some results by Feige and Ofek~\cite{ar:FeigeOfek2005} on the spectrum of the 
adjacency matrix of $G(n,p)$. 

\subsubsection*{The adjacency matrix of $G(n,p)$ and its spectral gap} 
In brief, Feige and Ofek~\cite{ar:FeigeOfek2005} showed that the second largest eignvalue of $G(n,p)$ is 
$O(\sqrt{np})$, provided that $np = \Omega (\log n)$. Let $A$ be the adjacency matrix of $G(n,p)$. 
If $G(n,p)$ were regular, then the all-1s vector $\mathbf{1}$ would span the eigenspace of the leading eigenvalue
and the above result would imply that $\langle Af,f\rangle = O(\sqrt{np}) \langle f,f\rangle$ for any function $f$ on the vertex set of $G(n,p)$ such that $f \perp \mathbf{1}$. However, $G(n,p)$ is not regular but almost regular 
in the sense that for any $\eps >0$ w.h.p. most vertices have degrees $np(1\pm \eps)$. 
Feige and Ofek~\cite{ar:FeigeOfek2005} proved that despite this, the bound on this quadratic form still 
holds w.h.p. 

To state these more precisely, let 
$S=\{ f : [n] \to \R : f \perp \mathbf{1} , \langle f , f\rangle \leq 1\} $ and fixing $0 < \delta <1$ we let
$$ T= \{ f : [n] \to \R : f \in \left\{  \frac{\delta}{\sqrt{n}} \Z \right\}^{[n]} , \ \langle f , f\rangle \leq 1\}. $$
The main theorem in~\cite{ar:FeigeOfek2005} is as follows. 
\begin{thm}[Theorem 2.5 in~\cite{ar:FeigeOfek2005}] \label{thm:FeigeOfek_main} 
Let $A$ be the adjacency matrix on $G(n,p)$, where $c_0 \frac{\log n}{n} \leq  p \leq \frac{n^{1/3}}{n(\log n)^{1/3}}
$. For every $c>0$, there exists $c'>0$ such that with probability at least $1 - n^{-c}$, for any 
$f,f' \in T$ we have 
$$| \langle Af, f'\rangle | \leq c' \sqrt{np}. $$
\end{thm}
Furthermore, the authors state and prove the following claim. 
\begin{clm} [Claim 2.4 in~\cite{ar:FeigeOfek2005}]\label{clm:Feige-Ofek 2.4} 
Suppose that for some $c >0$ we have $|\langle Af , f'\rangle |< c \sqrt{np}$, for any $f,f' \in T$. 
Then for any $f \in S$, we have $\langle Af, f \rangle < \frac{c}{(1-\delta)^2} \sqrt{np}$.  
\end{clm}
From the above two statements the following result is deduced. 
 \begin{thm} \label{thm:adj_Gnp} 
Let $A$ be the adjacency matrix on $G(n,p)$, where $c_0 \frac{\log n}{n} \leq  p \leq \frac{n^{1/3}}{n(\log n)^{1/3}}
$. For every $c>0$, there exists $c''>0$ such that with probability at least $1 - n^{-c}$, for any 
$f \in S$ we have 
$$ \langle Af, f \rangle \leq c'' \sqrt{np}. $$
\end{thm}

\subsubsection*{Deducing the lower bound on $\lambda (Y(n,p;d))$.}

We apply this in our setting, recalling that for any $\tau \in Y^{(d-2)}$, the link graph 
$\lk \tau$ is distributed as $G(n-d+1,p)$ with $p = \frac{(1+\eps) d \log n}{n}$. 
Taking $c=d$ in Theorem~\ref{thm:adj_Gnp}, we deduce that for some constant $c_d''>0$ 
with probability at least $1- n^{-d}$ for any
$f_\tau \in Z_0 (\lk \tau)$, we have 
$$\langle A_{\lk \tau} f_\tau, f_\tau \rangle \leq c_d'' \sqrt{np} \langle f_\tau, f_\tau \rangle. $$
The union bound (over all $\tau \in Y^{(d-2)}$) implies that for any $f\in Z_{d-1}(Y)$ we have w.h.p. 
\begin{equation} \label{eq:adjacency_sum}  
\sum_{\tau \in Y^{(d-2)}} \langle A_{\lk}f_\tau, f_\tau \rangle \leq 
c_d''\sqrt{np} \sum_{\tau \in Y^{(d-2)}} \langle f_\tau, f_\tau \rangle \stackrel{Thm~\ref{thm:garland}~4.}{=} 
d c_d'' \sqrt{np} \langle f, f\rangle.
\end{equation}
Using the lower bound of~\eqref{eq:Ds_lowerbound} and the upper bound of~\eqref{eq:adjacency_sum}, 
Equation~\eqref{eq:Laplace_decomp} yields that w.h.p. for any $f\in Z_{d-1} (Y)$ we have 
$$ \langle \Delta^+ f, f \rangle \geq \left(\delta (Y) - d c_d'' \sqrt{np} \right) \langle f, f \rangle.$$
But Lemma~\ref{lem:minDegree} states that $\delta (Y) \geq (1+ \eps) a d \log n - C\sqrt{\log n}$ w.h.p. 
for some $C>0$, where $a=a(\eps)$ is the solution of~\eqref{eqn:aDefn}. 
Hence, for some $C'>0$ w.h.p. for any $f \in Z_{d-1} (Y)$ such that $f \not = 0$ we have 
$$ \frac{ \langle \Delta^+ f, f \rangle}{  \langle f, f \rangle} \geq \delta (Y) - C' \sqrt{\log n}. $$
Therefore, w.h.p.
$$\lambda (Y) \geq  \delta (Y) - C' \sqrt{\log n}.$$
\end{proof}

\section{Combinatorial expansion} \label{sec:conductance}
Our lower bound on $\Phi_Y$ will rely on a lower bound on the edge expansion of the graph $(Y^{(d-1)}, E^{(d-1)})$, where the edge set $E^{(d-1)}$  consists of those distinct pairs $\sigma , \sigma' \in Y^{(d-1)}$ for which there exists 
a $d$-face $\rho \in Y^{(d)}$ which contains both $\sigma$ and $\sigma'$. We write $\sigma \sim \sigma'$. 
Hence, given a subset $S \subset Y^{(d-1)}$, we would like to bound from below the number of $d$-faces which
contain at least one $d-1$-face in $S$ and a $d-1$-face not in $S$. In other words, we would like to provide 
a lower bound on the number of $d$-faces which are potential \emph{exits} for a random walk that starts 
inside $S$. 

To express these more precisely, we introduce some relevant notation. 
For a $d$-dimensional complex $Y$ and $0 \leq k \leq d-1$, given $S \subset Y^{(k)}$ let
\[
 \partial^+ S = \{\rho \subset  Y^{(k+1)} \ : \ \mbox{there exists \ } \sigma \in S \mbox{ such that } \sigma \subset 
 \rho \}.
\]
Analogously to the oriented case, for $\sigma \in Y^{(k)}$ where $1\leq k\leq d$ we define
\[
 \partial \sigma = \{\tau \in Y^{(k-1)} : \tau \subset \sigma \}.
\] 

For $S\subset Y^{(d-1)}$ we let $\overline{S} = Y^{(d-1)} \setminus S$. We write
\begin{eqnarray*}
Q(S, \overline{S}) &=& \sum_{\sigma \in S} \sum_{\sigma' \in \overline{S} : \sigma' \sim \sigma} \pi(\sigma) \cdot \frac{1}{d \cdot \dgr{\sigma}} =  \frac{1}{d(d+1)\cdot |Y^{(d)}|}\sum_{\sigma \in S} \sum_{\sigma' \in \overline{S} : \sigma' \sim \sigma} \dgr{\sigma} \cdot \frac{1}{\dgr{\sigma}} \\
&=& \frac{1}{d(d+1) \cdot |Y^{(d)}|}\sum_{\sigma \in S} \sum_{\sigma' \in \overline{S} : \sigma' \sim \sigma} 1.
\end{eqnarray*}
With $B_S = \{ \sigma \in \partial^+ S : \partial \sigma \subset S \}$, 
we bound
$$ \sum_{\sigma \in S} \sum_{\sigma' \in \overline{S} : \sigma' \sim \sigma} 1 \geq |\partial^+S  \setminus B_S|.$$
The latter is the number of $d$-faces that are exits out of the set $S$. 

Furthermore, 
$$ \pi (S) = \frac{\sum_{\sigma \in S} \dgr{\sigma}}{d \cdot |Y^{(d)}|} \leq \frac{d\cdot  | \partial^+S |}{(d+1) \cdot |Y^{(d)}|}< \frac{| \partial^+S|}{|Y^{(d)}|}.$$
Hence, 
$$ \Phi_Y(S) \geq \frac{Q(S,\overline{S})}{\pi (S)} \geq \frac{1}{d(d+1)} \cdot \frac{|\partial^+S\setminus B_S|}{| \partial^+S|},$$
whereby 
\begin{equation} \label{eq:conductance_lower} 
\Phi_Y \geq \frac{1}{d(d+1)} \cdot \min_{S\subset Y^{(d)}: 0 < \pi (S) \leq \frac12} \frac{|\partial^+S\setminus B_S|}{| \partial^+S|}.
\end{equation}

In our proof, we will in fact use an upper bound on the number of $d$-faces in $\partial^+ S$ which are not exits. 
These are $d$-faces whose $d-1$-subsets are all faces belonging to $S$. To bound their number from above,  
we will use a weak version of the Kruskal-Katona theorem. This provides an upper bound on the number of 
complete subgraphs on a hypergraph with a given number of hyperedges. To apply this to our context, we consider the hypergraph spanned by the $d-1$-faces in $S$. The number of the complete $d+1$-subhypergraphs of this hypergraph is an upper bound on the number of $d$-faces in $\partial^+S$ which 
are not exits. 

Note that two $d-1$-faces $\sigma, \gamma \in Y^{(d-1)}$ which are \emph{neighbours}, that is, 
$\sigma \sim \sigma'$, satisfy $|\sigma \cap \gamma| = d-1$. Given $X \subseteq Y^{(d-1)}$, we say that $X$ is \emph{tightly connected} if for any distinct $\sigma, \sigma' \in X$ there exists $m \geq 1$ and a sequence $\sigma = \delta_0, \ldots, \delta_m = \sigma' \in X$ such that $\delta_i \sim \delta_{i+1}$ for all $0 \leq i \leq m-1$.
As we observed earlier, the random walk we consider walks over tightly connected sets. 
To bound the quantity on the right-hand side of~\eqref{eq:conductance_lower}, 
we start with the following theorem, which 
considers only subsets $S \subset Y^{(d-1)}$ which are tightly connected and $\pi (S) \leq 1/2$.  
In fact, it suffices to consider those subsets $S$ with $|S| \leq \frac{1}{2}{n \choose d}$. 
To see this, recall Lemma~\ref{lem:minDegree} implies that w.h.p. $\deg (\sigma) \geq d+1$. 
Hence, if $\pi (S) \leq 1/2$, then
$$ (d+1) |S| \leq \sum_{\sigma \in S} \dgr{\sigma}\leq \frac{d+1}{2} |Y^{(d)}| \leq \frac{d+1}{2} {n \choose d}. $$
This directly implies that $|S| \leq \frac{1}{2} {n \choose d}$. 
\begin{thm}
\label{thm:mainExpansion}
Let $Y = Y(n,p;d)$ where $np=(1+\eps)d\log n$ for $\eps>0$ fixed. There exists $\delta > 0$ such that w.h.p. the following holds. For any tightly connected set $S \subset Y^{(d-1)}$ with $|S| \leq \tfrac{1}{2} \binom{n}{d}$ and
\[
 B_S = \{ \sigma \in \partial^+ S : \partial \sigma \subset S \}
\]
we have
\begin{equation}
\label{eqn:mainExpansion}
 | (\partial^+S  \setminus B_S)\cap Y^{(d)} | \geq \delta | \partial^+S \cap Y^{(d)}|.
\end{equation}
\end{thm}

\begin{proof}
We shall assume that $n$ is large enough for the estimates in the proof to hold. Given $S \subset Y^{(d-1)}$ and $1 \leq i \leq d+1$, let
\[
 F_i(S) = \{ \rho \in \partial^+S : | \partial \rho \cap S | = i \},
\]
and set $f_i(S) = |F_i(S)|$. Denoting $|S| = m$, by double counting, we have that
\begin{equation}
\label{eqn:fisum}
\sum_{i=1}^{d+1} i f_i(S) = m(n-d).
\end{equation}

For any $t \geq 2$, let $K_{t+1}^{(t)}$ be the complete $t$-uniform hypergraph on $t+1$ vertices, and for a $t$-uniform hypergraph $G$, let $K_{t+1}^{(t)}(G)$ denote the number of copies of $K_{t+1}^{(t)}$ in $G$. Note that in fact we have $f_{d+1}(S) = K_{d+1}^{(d)}(S)$, i.e., $f_{d+1}(S) = |B_S|$. Hence, we shall use the following weak form of Kruskal-Katona theorem (see Lov{\'a}sz \cite{lovaszcpe}).
\begin{thm}
\label{thm:kk}
Suppose $r \geq 1$ and $G$ is an $r$-uniform hypergraph with
\[
m = \binom{x_m}{r} = \frac{x_m(x_m-1)\ldots(x_m-r+1)}{r!}
\]
hyperedges, for some real number $x_m \geq r$. Then $K^{(r)}_{r+1}(G) \leq \binom{x_m}{r+1}$, with equality if and only if $x_m$ is an integer and $G = K^{(r)}_{x_m}$.
\end{thm}
Note that we can rewrite \eqref{eqn:fisum} to get
\begin{equation}
\label{eqn:fiineq}
\sum_{i=1}^{d} f_i(S) \geq \frac{1}{d} \left ( m(n-d) - (d+1)f_{d+1}(S) \right ) = \frac{nm}{d} \left ( 1-\frac{d}{n} - \frac{(d+1)f_{d+1}(S)}{nm} \right ) .
\end{equation}

Observe that $| \partial^+S \setminus B_S|$ follows the binomial distribution $\Bin(\sum_{i=1}^{d} f_i(S),p)$. 
By Theorem \ref{thm:kk} we have
\[
\frac{f_{d+1}(S)}{nm} \leq \frac{1}{n} \frac{\binom{x_m}{d+1}}{\binom{x_m}{d}} = \frac{x_m-d}{n(d+1)}.
\]
However, by assumption $x_m$ satisfies $m = \binom{x_m}{d} \geq \frac{(x_m-d)^d}{d!}$; thus, we obtain $x_m-d \leq (md!)^{1/d}$. This yields
\begin{equation}
\label{eqn:fd+1bound}
\frac{f_{d+1}(S)}{nm} \leq \frac{(md!)^{1/d}}{n(d+1)}.
\end{equation}
By \eqref{eqn:fiineq} and \eqref{eqn:fd+1bound} we obtain
\begin{equation}
\label{eqn:fi2ndineq}
\sum_{i=1}^{d} f_i(S) \geq \frac{nm}{d} \left ( 1-\frac{d}{n} - \frac{(md!)^{1/d}}{n} \right ) = \frac{nm}{d} \left ( 1 - \frac{(md!)^{1/d}}{n} - o(1) \right ) .
\end{equation}

For a fixed $S$ with $|S| = m$ we will bound from above the probability of the event that $| \partial^+S  \setminus B_S | \leq \delta | \partial^+S |$. It can be easily seen that $| \partial^+S \cap Y^{(d)} |$ is stochastically 
dominated by a random variable that has the $\Bin(nm,p)$ distribution. So by \eqref{eqn:ChernoffUpper} we have that
\begin{align}
\label{eqn:RHSbound}
 \Prob{| \partial^+S \cap Y^{(d)}  | \leq 3nmp} & \geq \Prob{\Bin(nm,p) \leq 3nmp} \nonumber \\
  & \geq 1 - \exp \left ( -\frac{4nmp}{3} \right ) \nonumber \\ 
  & = 1-\exp \left ( - \frac{(1+\eps) 4 d m \log n}{3}\right ).
\end{align}
Hence, we will now aim to bound the probability that $| \partial^+S  \setminus B_S | \leq \delta 3nmp$. We are interested in sets $S$ of $(d-1)$-dimensional faces of size
\[
m \leq \frac{1}{2}\binom{n}{d} \leq \frac{n^d}{2d!},
\]
which implies that
\[
\frac{nm}{d} \left ( 1 - \frac{(md!)^{1/d}}{n} - o(1) \right ) \geq \frac{nm}{d} \left ( 1 -2^{-1/d} - o(1) \right ).
\]
Hence by \eqref{eqn:fi2ndineq}, the random variable $| (\partial^+S \setminus B _S)\cap Y^{(d)}|$ is stochastically
bounded from below by a random variable distributed as 
$\Bin \left ( \lceil \frac{nm}{d} \left ( 1 -2^{-1/d} - o(1) \right ) \rceil, p\right )$.
Thus, 
\begin{align*}
 \Prob{| (\partial^+S \setminus B _S)\cap Y^{(d)}| \leq \delta 3nmp} \leq \Prob{\Bin \left (\left\lceil \frac{nm}{d} \left ( 1 -2^{-1/d} - o(1) \right )\right\rceil, p\right ) \leq \delta 3nmp }.
\end{align*}
Let 
\[
 \mu =\left\lceil \frac{nm}{d} \left ( 1 -2^{-1/d} - o(1) \right ) \right\rceil p\geq  (1+\eps) \left ( 1 -2^{-1/d} - o(1) \right ) m \log n.
\]
We have
\begin{align*}
  \frac{\mu - \delta 3nmp}{\mu} \geq 1 - \frac{3 \delta d}{1 -2^{-1/d}} - o(1) > 0
\end{align*}
for
\[
\delta <  \frac{ 1 -2^{-1/d} }{ 3 d }.
\]
Hence, by the Chernoff bound  \eqref{eqn:ChernoffLower} we obtain
\begin{equation}
\label{eqn:LHSbound}
 \Prob{|( \partial^+S \setminus B_S) \cap Y^{(d)} | \leq \delta 3nmp} \leq \exp \left ( - \frac{\left ( 1 -2^{-1/d} - 3 \delta d - o(1) \right )^2 }{2} \mu \right ).
\end{equation}

We shall consider two cases: 1. $n^{d(1-\alpha)} \leq m \leq \tfrac{1}{2} \binom{n}{d}$, and 2. $m < n^{d(1-\alpha)}$, for some $0 < \alpha < 1$ to be specified later.

Let us consider the case $m \geq n^{d(1-\alpha)}$ first. The number of sets $S$ of size $m$ is then at most
\begin{align*}
 \binom{\binom{n}{d}}{m} & \leq \left ( \frac{e \binom{n}{d}}m \right )^m \leq \exp \left ( m + m \log \frac{n^d}{m} \right ) \leq \exp \left ( m + m \log n^{\alpha d} \right ) \\
 & = \exp \left ( (1+ o(1)) \alpha d m \log n \right ).
\end{align*}
There are at most $n^d$ possible values of $m$ in that region. Hence, if
\[
\alpha < \frac{(1+\eps) \left ( 1 -2^{-1/d} - o(1) \right ) \left ( 1 -2^{-1/d} - 3 \delta d - o(1) \right )^2 }{2d},
\]
then by \eqref{eqn:RHSbound}, \eqref{eqn:LHSbound}, and the union bound, we have that \eqref{eqn:mainExpansion} holds with probability $1-o(1)$ for all sets $S$ of size at least $n^{d(1-\alpha)}$.

Next, we consider the case when $m < n^{d(1-\alpha)}$, i.e., when $m^{1/d}/n < n^{-\alpha}$. By \eqref{eqn:fd+1bound} we have
\[
  |B_S| = f_{d+1}(S) \leq mn \frac{(md!)^{1/d}}{(d+1)n} \leq mn^{1-\alpha} \frac{d}{d+1} \leq mn^{1-\alpha}.
\]
So $|B_S \cap Y^{(d)}|$ is stochastically bounded from above by a random variable that 
distributed as $\Bin (\lfloor mn^{1-\alpha}\rfloor ,p)$.
Let $k_0 = \lceil 3mn^{1-\alpha/2}p \rceil$. By the Chernoff bound~\eqref{eqn:ChernoffUpper} we have that
\begin{align}
\label{eqn:smallB}
 \Prob{| B_S \cap Y^{(d)}| \geq k_0 } & \leq \Prob{\Bin(\lfloor mn^{1-\alpha}\rfloor,p) \geq (1+2n^{\alpha/2}) mn^{1-\alpha}p} \nonumber \\
  & \leq \exp \left ( -\frac{4 n^\alpha n^{1-\alpha} mp}{3} \right ) \nonumber \\ 
  & = \exp \left ( - \frac{(1+\eps) 4d m \log n}{3}\right ).
\end{align}
By \eqref{eqn:fi2ndineq} and the fact that $m=o(n^d)$ we obtain
\[
\sum_{i=1}^{d} f_i(S) \geq \left ( 1 - o(1) \right ) \frac{nm}{d},
\]
which implies that $| ( \partial^+S \setminus B_S )\cap Y^{(d)} |$ stochastically dominates the $\Bin \left ( \frac{nm}{d} \left ( 1 - o(1) \right ), p\right )$ distribution. Since
\[
k_0 = o \left ( \E{\Bin \left ( \frac{nm}{d} \left ( 1 - o(1) \right ), p\right )} \right ),
\]
we then have that
\begin{align}
\label{eqn:largeShadow}
 \Prob{| (\partial^+S \setminus B_S) \cap Y^{(d)} | \leq k_0} & \leq \Prob{\Bin \left ( \frac{nm}{d} \left ( 1 - o(1) \right ), p\right ) \leq  k_0 } \nonumber \\
  & = \sum_{k=0}^{k_0} \Prob{\Bin \left ( \frac{nm}{d} \left ( 1 - o(1) \right ), p\right ) = k } \nonumber \\
  & \leq 2 k_0 \Prob{\Bin \left ( \frac{nm}{d} \left ( 1 - o(1) \right ), p\right ) = k_0 } \nonumber \\
  & \leq 2 k_0 \binom{ \frac{nm}{d} \left ( 1 - o(1) \right )}{k_0} p^{k_0} (1-p)^{\frac{nm}{d} \left ( 1 - o(1) \right )-k_0} \nonumber \\
  & \leq 2 k_0 \left ( \frac{ 3nmp }{dk_0} \right )^{k_0} (1-p)^{\frac{nm}{d} \left ( 1 - o(1) \right )-k_0} \nonumber \\
  & \leq 2 k_0 \left ( \frac{ 3nmp }{dk_0} \right )^{k_0} \exp \left ( - \frac{nmp}{d} \left ( 1 - o(1) \right )+k_0p \right ) \nonumber \\
  & = 2 k_0 \left ( \frac{ 3nmp }{dk_0} \right )^{k_0} \exp \left ( - (1 - o(1) ) (1+\eps) m \log n  \right ).
\end{align}
Now,
\begin{align}
\label{eqn:fewCases}
 \left ( \frac{ 3nmp }{d} \right )^{k_0} & = \exp \left ( k_0 \log \frac{ 3nmp }{d} \right ) \nonumber \\
  & \leq \exp \left ( \frac{4 d m \log n}{ n^{\alpha/2}} \log (m \log n)  \right ) \nonumber \\
  & = \exp \left ( o ( m \log n) \right ).
\end{align}
Combining \eqref{eqn:smallB}, \eqref{eqn:largeShadow} and \eqref{eqn:fewCases} together we obtain
\begin{align}
\label{eqn:boundSmallSets}
& \Prob{| (\partial^+S \setminus B_S) \cap Y^{(d)} | \geq | \partial^+S  | / 2 } \nonumber \\
  & \qquad \geq \Prob{| B_S \cap Y^{(d)}| \leq k_0 \mbox{ and } |  (\partial^+S \setminus B_S) \cap Y^{(d)}| 
  \geq k_0 } \nonumber \\
   & \qquad \geq 1 - \exp \left ( - (1 - o(1) ) (1+\eps) m \log n  \right ).
\end{align}

The bound in \eqref{eqn:boundSmallSets} is not strong enough to apply it with the union bound over all sets $S$ of size $m$. However, we assumed that $S$ is tightly connected and we will now exploit this assumption. We bound the number of tightly connected sets $S \in Y^{(d-1)}$ with $|S| = m$ as follows. Order the set $Y^{(d-1)}$ of $(d-1)$-faces in an arbitrary way; for example, identifying every face $\sigma$ with an ordered tuple $(\sigma_0, \ldots, \sigma_{d-1})$ where $\sigma_0 < \ldots < \sigma_{d-1}$, we could order the faces increasingly in the lexicographic order. We can pick the first face $\sigma \in S$ in $\binom{n}{d}$ many ways. We then perform a breadth-first-search on $S$: we first find all neighbours of $\sigma$, i.e., faces that share $d-1$ vertices with $\sigma$. Exploring these faces according to the selected order, we then find all yet unexplored faces that share $d-1$ vertices with consecutive neighbours of $\sigma$. Then, we move to the second neighbourhood of $\sigma$ and find all of their still unexplored neighbours. Since, having picked $\sigma$, we have to discover a total of $m-1$ faces and these will be first found as one of the offspring of one of $m$ faces, we see that there are at most $\binom{2m-2}{m-1} \leq 4^m$ many ways to assign the numbers of offspring to consecutive faces 
(a collection of $m$ non-negative integers which sum up to $m-1$).

Any $(d-1)$-dimensional face $\sigma$ has at most $dn$ neighbours, as we have $d$ vertices in $\sigma$ we can drop, and at most $n$ vertices not in $\sigma$ can be added to form the neighbour. Hence, for any choice of the numbers of neighbours first explored by consecutive faces, there are at most $(dn)^m$ ways to pick these neighbours. Thus, the number of tightly connected sets of size $m$ is at most
\[
4^m (dn)^m = \exp((1+o(1))m \log n).
\]
As again the number of values of $m$ we have to consider is at most $n^d$, by \eqref{eqn:boundSmallSets} and the union bound we see that with probability $1-o(1)$, \eqref{eqn:mainExpansion} holds for all sets $S$ with $|S| \leq n^{d(1-\alpha)}$. This completes the proof of Theorem~\ref{thm:mainExpansion}.

\end{proof}

We can now easily show that the assumption that $S$ is tightly connected can be dropped in Theorem~\ref{thm:mainExpansion} yielding a lower bound on $\Phi_{Y(n,p;d)}$ and completing the
proof of Theorem~\ref{thm:conductance}.
\begin{cor}
Let $Y = Y(n,p;d)$ with $np=(1+\eps)d\log n$ for $\eps>0$ fixed. 
There exists $\delta > 0$ such that w.h.p. the following holds. For any set $S \subset Y^{(d-1)}$
we have
\begin{equation}
\label{eqn:mainAllSets}
 | (\partial^+S  \setminus B_S)\cap Y^{(d)} | \geq \delta | \partial^+S \cap Y^{(d)}|.
\end{equation}
\end{cor}
\begin{proof}
First, observe that if $U, V$ are distinct maximal tightly connected sets in $Y^{(d-1)}$ then $\partial^+U \cap \partial^+ V = \emptyset$. Indeed, if $\rho \in \partial^+ U \cap \partial^+ V$ then there exist $\sigma_1 \in U, \sigma_2 \in V$ such that $|\rho \cap \sigma_1| = |\rho \cap \sigma_2| = d$, but that implies $|\sigma_1 \cap \sigma_2| = d-1$, so $\sigma_1, \sigma_2$ are incident; a contradiction.

Hence, let $S$ be a union of maximal tightly connected sets $S_1, \ldots, S_p$. For $1 \leq i \leq p$, let
\[
 B_i = \{ \rho \in \partial^+ S_i : \partial \rho \subset S_i \}
\]
and $B= \cup_{i=1}^p B_i$.
Set $\alpha_i = |(\partial^+S_i \cap Y^{(d)}) \setminus B_i |$ and $\beta_i = | \partial^+S_i \cap Y^{(d)} |$. By Theorem \ref{thm:mainExpansion} with probability $1-o(1)$ we have $\alpha_i / \beta_i \geq \delta$ for all $i$.

Now observe that $a,b,c,d > 0$ and $a/b > c/d$ implies $a> cb/d$ and therefore
\[
\frac{a+c}{b+d} > \frac{c\frac{b}{d}+c}{b+d} = \frac{c\frac{b+d}{d}}{b+d} = \frac{c}{d}.
\]
Thus by the disjointness of the sets $\partial^+S_1, \ldots, \partial^+ S_p$ we deduce that with probability $1-o(1)$ we have
\[
\frac{ | (\partial^+S \cap Y^{(d)}) \setminus B | }{ | \partial^+S \cap Y^{(d)} |} \geq \min_{1 \leq i \leq p} \frac{\alpha_i}{\beta_i} \geq \delta.
\]
This completes the proof.
\end{proof}

\section{Conclusions} This paper is a study of various measures of expansion in the Linial-Meshulam random complex $Y(n,p;d)$ past the cohomological connectivity threshold. We considered the spectral gap of the combinatorial 
Laplace operator and showed that w.h.p. it is very close to the the Cheeger constant associated with the simplicial complex. Furthermore, we showed that both quantities are w.h.p. very close to the minimum co-degree of the random simplicial complex. We determined explicitly the latter using the large deviations theory of the binomial distribution.  

Finally, we considered a random walk on the $d-1$-faces of the random simplicial complex, which generalises the standard random walk on graphs.  In particular, we considered the conductance of such a random walk and showed that w.h.p. it is bounded away from zero. 

The above results were obtained for $p$ such that $np = (1+\eps )d \log n$, for any $\eps >0$ fixed. 
Our proofs seem to work when $\eps  =\eps (n) \to 0$ as $n\to \infty$ slowly enough. 
A natural next step would be to consider these quantities for $p$ that is closer to the threshold $d \log n /n$. Indeed, the supercritical regime is for $p$ such that $np = d \log n  + \omega (n)$, where $\omega(n) \to \infty$ as 
$n \to \infty$ arbitrarily slowly. We believe it would be interesting to extend the analysis to this range of $p$ as well. 
This would complete the picture of the evolution of the expansion properties of $Y(n,p;d)$. 

\section{Acknowledgements} We would like to thank Eoin Long for suggesting to us the use of the Kruskal-Katona theorem in the context of bounding the combinatorial expansion of $Y(n,p;d)$.

\bibliographystyle{plain}
\bibliography{simp_expansion}

\end{document}